\documentclass[12pt,letterpaper,reqno]{amsart}

\usepackage{amsthm,amssymb,mathrsfs}
\usepackage{amsmath,amscd}
\usepackage{graphicx} 
\usepackage{xcolor}
\usepackage{pinlabel}
\usepackage{fourier}
\usepackage[breaklinks,colorlinks,citecolor=blue,linkcolor=blue,
 urlcolor=teal]{hyperref} 

\usepackage[margin=3.4cm]{geometry}

\theoremstyle{plain}
\newtheorem{theorem}[equation]{Theorem} 
\newtheorem{proposition}[equation]{Proposition}
\newtheorem{lemma}[equation]{Lemma}
\newtheorem{corollary}[equation]{Corollary}
\newtheorem*{thma}{Theorem A}

\theoremstyle{definition}
\newtheorem{definition}[equation]{Definition}
\newtheorem{remark}[equation]{Remark}
\newtheorem{remarks}[equation]{Remarks}

\newcommand{\co}{\colon \thinspace}

\newcommand{\Z}{{\mathbb Z}}

\newcommand{\abs}[1]{\lvert {#1} \rvert}
\newcommand{\aabs}[1]{\lVert {#1} \rVert}
\renewcommand{\leq}{\leqslant}
\renewcommand{\geq}{\geqslant}
\renewcommand{\epsilon}{\varepsilon}
\renewcommand{\phi}{\varphi}
\newcommand{\N}{{\mathbb N}}
\newcommand{\Ss}{{\mathscr S}}
\newcommand{\Is}{{\mathscr I}}
\newcommand{\As}{{\mathscr A}}
\newcommand{\Bs}{{\mathscr B}}

\newcommand{\Ps}{{\mathscr P}}

\renewcommand{\preceq}{\preccurlyeq}
\renewcommand{\succeq}{\succcurlyeq}

\DeclareMathOperator{\id}{id}
\DeclareMathOperator{\area}{Area}

\numberwithin{equation}{section}

\makeatletter
  \def\tagform@#1{\maketag@@@{%
   \textbf{(\ignorespaces#1\unskip\@@italiccorr)}}}%
   \renewcommand{\eqref}[1]{\textup{\maketag@@@{(\ignorespaces%
        {\ref{#1}}\unskip\@@italiccorr)}}}
\makeatother

\begin{document}

\title[Snowflake geometry in CAT(0) groups]{Snowflake geometry in 
CAT(0) groups}

\author{Noel Brady}
\address{Mathematics Department\\
        University of Oklahoma\\
        Norman, OK 73019\\
        USA}
\email{nbrady@math.ou.edu\\forester@math.ou.edu}
\author{Max Forester}

\begin{abstract} 
We construct CAT(0) groups containing subgroups whose Dehn functions are
given by $x^s$, for a dense set of numbers $s \in [2, \infty)$. This
significantly expands the known geometric behavior of subgroups of
CAT(0) groups. 
\end{abstract}

\maketitle

\thispagestyle{empty}

\section{Introduction}
\label{sec:intro}

In this paper we are concerned with the geometry of subgroups of
CAT(0) groups. Such subgroups need not be CAT(0) themselves, and in
fact this realm includes many instances of exotic or unusual
group-theoretic behavior. For example, there exist finitely presented
subgroups of CAT(0) groups having unsolvable membership and conjugacy
problems \cite{BBMS2,Bridson:raags}, or possessing infinitely many
conjugacy classes of finite subgroups (implicitly in \cite{FM},
explicitly in \cite{LN,BCD}). Many well known
examples of groups having interesting homological properties, such 
as the Stallings--Bieri groups or other Bestvina--Brady kernels,
arise naturally as subgroups of CAT(0) groups. 

Regarding the geometry, specifically, of subgroups of CAT(0) groups,
only a few types of unusual behavior have been observed to date. One
common feature to all subgroups of CAT(0) groups is that cyclic
subgroups are always undistorted. This constraint immediately rules
out a great many examples from being embeddable into CAT(0)
groups. The examples found so far with interesting geometric
properties have been constructed using fairly sophisticated
techniques. Such examples include the hydra groups of \cite{DR},
which are CAT(0) groups possessing free subgroups with extreme
distortion; subgroups of CAT(0) groups having exponentional or
polynomial Dehn functions \cite{BBMS,Bridson:raags},
\cite{BRS,ABDDY,BGL}; and subgroups of CAT(0) groups having different
homological and ordinary Dehn functions \cite{ABDY}. 

Our goal in the present paper is to construct CAT(0) groups containing
subgroups that exhibit a wide range of isoperimetric behavior. 
Our main theorem is the following. 

\begin{thma}
Let $m, n$ be positive integers such that $\alpha = n\log_m(1 +
\sqrt{2}) \geq 1$. Then there exists a $6$--dimensional
\textup{CAT(0)} group $G$ which contains a finitely presented subgroup
$S$ whose Dehn function is given by $\delta_S(x) = x^{2\alpha}$. 
\end{thma}

Varying $m$ and $n$, the resulting isoperimetric exponents $2\alpha$
form a dense set in the interval $[2, \infty)$. Thus, the
\emph{isoperimetric spectrum} of exponents arising from subgroups of
CAT(0) groups resembles, in its coarse structure, that of finitely
presented groups generally \cite{BrBr}. 

\setcounter{tocdepth}{1}
\tableofcontents

The group $S$ in the main theorem is constructed with reference to
certain parameters; it is denoted $S_{T,n}$ where $T$ is a finite tree
and $n$ is an integer. It will embed into a CAT(0) group denoted
$G_{T,n}$. 

\subsection*{Comparison with prior constructions of snowflake groups}
The construction of the snowflake groups $S_{T,n}$ builds on the
methods used in \cite{BBFS} to construct groups with specified power
functions as Dehn functions. The latter groups do not embed into any
CAT(0) group, for the simple reason that they contain distorted cyclic
subgroups. The same is true for the groups constructed in \cite{SBR},
and many other groups having interesting Dehn functions. Indeed, it is
the presence of precisely distorted cyclic subgroups that enables the
computation of the Dehn functions in \cite{BBFS} in the first place. 

The basic structure of the groups in \cite{BBFS} is that of a graph of
groups, with vertex groups of cohomological dimension $2$, and
infinite cyclic edge groups. The lower bound for the Dehn function was
easy to establish, using asphericity of the presentation $2$--complex. 
The distortion of the cyclic edge groups was then precisely determined,
and this distortion estimate led to a matching upper bound for the Dehn
function of the ambient group. 

If one hopes to embed examples into CAT(0) groups, one cannot have
distorted cyclic subgroups. The groups $S_{T,n}$ we construct here are
fundamental groups of graphs of groups in which the edge groups are
\emph{free groups} of rank $2$. Moreover, the vertex groups are groups
of cohomological dimension $3$, and therefore $S_{T,n}$ does not admit
an aspherical presentation. For this reason, establishing the lower
bound for the Dehn function of $S_{T,n}$ requires some effort. 

The change to free edge groups introduces some major challenges. 
In \cite{BBFS}, given any element of an edge group, conjugation by the
appropriate stable letter resulted in a word $r$ times longer, for a 
uniform factor $r$. In the Cayley $2$--complex, each edge
space had a well defined long side and short side. A key geometric
property was that no geodesic could pass from the short side to the long
side of any edge space, and back again. It thus became possible to
determine the large scale behavior of geodesics, which in turn led to
very explicit distance estimates in the group. 

In the groups $S_{T,n}$ defined here, the edge groups have monodromy
modeled on a hyperbolic free group automorphism. Under this
automorphism, some words get longer and some get shorter. It is no
longer possible to constrain the behavior of geodesics relative to the
Bass-Serre tree for the graph of groups structure as in
\cite{BBFS}. This local stretching and compressing behavior of the
automorphism is a phenomenon that was similarly faced in the work of
Bridson and Groves in \cite{BrGr}. To accommodate this behavior, we
rely on a coarser and more robust approach. We begin with the
framework used in \cite{BrBr}, and develop additional techniques for
handling the lack of local control over geodesics. These ideas are
described in more detail in the subsection ``Edge group distortion''
near the end of this Introduction.

\subsection*{The embedding trick}
The method we use for embedding $S_{T,n}$ into a CAT(0) group is based
on a twisting phenomenon for graphs of groups. Suppose $\alpha \co G_e
\to G_v$ is the inclusion map from an edge group to a vertex group,
where $e$ is incident to $v$. If one replaces $\alpha$ by $\phi
\circ\alpha$, where $\phi$ is an automorphism of $\alpha(G_e)$, then
typically one obtains a very different fundamental group. However, if
$\phi$ is the restriction of an \emph{inner automorphism} of $G_v$,
then the fundamental group remains unchanged. 

The graph of groups structure of the group $S_{T,n}$ incorporates in
an essential way the ``twisting automorphisms'' $\phi$ just described;
properties of these automorphisms influence strongly the geometry of
$S_{T,n}$. The presence of twisting accounts for the lack of
non-positive curvature and the possibility of larger-than-quadratic
isoperimetric behavior. 

In order to embed $S_{T,n}$ into a larger group $G_{T,n}$, we give the
latter the structure of a graph of groups that is very similar to that
of $S_{T,n}$. It has the same underlying graph, somewhat larger
edge and vertex groups, and inclusion maps which make use of 
twisting automorphisms closely related to those in $S_{T,n}$. This is to
ensure the existence of a morphism of graphs of groups which induces a
homomorphism $S_{T,n} \to G_{T,n}$. Injectivity of this homomorphism
is established using a criterion due to Bass \cite{Bass}; see Lemma
\ref{basslemma}. 

The basic trick now is that the vertex group in $G_{T,n}$ has been
enlarged specifically to arrange that the twisting automorphisms
become \emph{inner}. Thus the twisting can be undone, without changing
the fundamental group. With no twisting, it becomes a simple matter to
put a CAT(0) structure on $G_{T,n}$, provided the enlarged vertex
group is already CAT(0). 

\subsection*{Bieri doubles and the embedding trick} Earlier we
referred to examples of subgroups of CAT(0) groups having exponentional
or polynomial Dehn functions. The examples from \cite{BBMS,BRS} made
use of an embedding theorem from \cite{BBMS}. The authors embed
certain doubles of groups into the direct product of a related group
and a free group; the latter product may then have a CAT(0) structure,
while the embedded subgroup has interesting geometric behavior. The
embedded subgroup is called a \emph{Bieri double}. 

It turns out that some instances of this method can be viewed as
instances of the embedding trick discussed above, resulting in an 
alternative way of looking at certain Bieri doubles. 

Here is an example, which includes the \cite{BRS} examples of
subgroups of CAT(0) groups having polynomial Dehn functions. Let $N$ be
any group with an automorphism $\phi$, 
and consider the double of $N \rtimes_{\phi} \Z$ along $N$: $(N
\rtimes \langle t_1 \rangle) \ast_N (N \rtimes \langle
t_2\rangle)$. There is a homomorphism 
\[ (N \rtimes \langle t_1 \rangle) \ast_N (N \rtimes \langle
  t_2\rangle) \ \to \ (N \rtimes \langle t \rangle)
\times \langle u, v\rangle\]
which sends $N$ to $N$, $t_1$ to $tu$, and $t_2$ to $tv$. The
arguments of \cite{BBMS} show that this map is an embedding. 

An alternative viewpoint is to note that the double $(N \rtimes
\langle t_1 \rangle) \ast_N (N \rtimes \langle t_2\rangle)$ has a
graph of groups decomposition $\As$ with underlying graph a figure-eight,
with all edge and vertex groups $N$, and with monodromy along each
loop given by $\phi$. The elements $t_1$ and $t_2$ are the two stable
letters. The larger group $(N \rtimes \langle t \rangle ) \times
\langle u, v\rangle$ has a similar decomposition with the same
underlying graph, with vertex and edge groups $N \rtimes \langle t
\rangle$, with monodromy the identity. In this case $u$ and $v$
are the two stable letters. 

This description of the larger group is the ``untwisted'' presentation
for it. If we change the monodromy to be by the automorphism
$\phi \times \id$ of $N \rtimes \langle t \rangle$, then the new graph
of groups $\Bs$ still has fundamental group $(N \rtimes \langle t \rangle)
\times \langle u, v \rangle$, since $\phi \times \id$ is conjugation
by $t$ in the vertex group $N \rtimes \langle t \rangle$. The two
stable letters for this new decomposition are $tu$ and $tv$. There is
an obvious morphism of graphs of groups $\As \to \Bs$ since $\phi
\times \id$ restricts to $\phi$ on $N$. The induced map on fundamental
groups is exactly the embedding displayed above. 

\subsection*{An overview of the construction}
Recall that our basic aim is to construct a pair of finitely presented
groups $S$ and $G$ such that $G$ is CAT(0), $S$ embeds into $G$, and $S$
has a specified Dehn function. The starting data needed for our
constructions is as follows. 

First choose a palindromic monotone automorphism $\phi \co \langle x, y
\rangle \to \langle x, y \rangle$ (see Section \ref{sec:prelim} for these
notions). Let $\lambda$ be the Perron-Frobenius eigenvalue of the
transition matrix of $\phi$. Also let $T$ be a finite tree with valence
at most $3$, and let $m = \abs{T}+1$ (here, $\abs{T}$ denotes the number
of vertices of $T$). Choose a positive integer $n$ such that $\alpha =
n\log_m(\lambda) \geq 1$. Based on these choices, we define finitely
presented groups $S_{T,n}$ and $G_{T,n}$. 

These groups are multiple HNN extensions of groups $V_T$ and $W_T$ 
respectively, with $m$ stable letters. In Section \ref{sec:groups} we
define these vertex groups; they are, themselves, fundamental groups of
graphs of groups with underlying graph $T$. The group $V_T$, in
particular, is required to have some very special properties, notably
the \emph{balancing property}. The structure needed to achieve this
property also dictates how $W_T$ should be constructed. 

Briefly, $V_T$ has vertex groups isomorphic to $F_2 \times F_2 \times
F_2$, and each such group has three designated \emph{peripheral
subgroups} which are used as edge groups; these edge groups are
isomorphic to $F_2$. They are the ``antidiagonal'' copies of $F_2$ in
each factor $F_2 \times F_2$ in each vertex group. 

The group $W_T$ is built in a similar way, with the free-by-cyclic
group $G = F_2 \rtimes_{\phi} \Z$ used in place of $F_2$; thus the
vertex groups are $G \times G \times G$, and the edge groups are
isomorphic to $G$. The groups $G$ possess $2$--dimensional CAT(0)
structures, by work of Tom Brady \cite{TBr}. 

Given the CAT(0) structure on $W_T$, the ambient group $G_{T,n}$ can
also be made CAT(0), as in the discussion of the Embedding Trick
above; see Section \ref{sec:cat0}. When building this CAT(0)
structure, we fix the free group automorphism $\phi$ to be a
particular one which is palindromic. Tom Brady's CAT(0) structure has
a symmetry that respects the palindromic nature of 
$\phi$. In the case of the other free-by-cyclic groups  
arising in \cite{TBr}, we believe that analogous symmetries exist (in
the palindromic case), but establishing this would take us too far
afield. It is for this reason that the exponents in the main theorem
involve $1+\sqrt{2}$ (which is $\lambda$ for this choice of
$\phi$). Apart from the CAT(0) statement for $G_{T,n}$, the rest of
the paper does not require any particular choice of $\phi$, and we
compute the Dehn functions of $S_{T,n}$ in terms of a general
(monotone, palindromic) $\phi$.

\subsection*{Corridor schemes and $\sigma$--corridors} There are several
places where we make use of corridor arguments in van Kampen
diagrams. There are many different types of corridors under
consideration simultaneously, and these different types are codified as
\emph{corridor schemes}. In particular, the group $V_T$ possesses a large
number of distinct corridor schemes, which are parametrized by maximal
segments $\sigma$ in a tree $\widehat{T}$. Each such segment determines a
corridor scheme, whose corridors are called
\emph{$\sigma$--corridors}. 

These corridors provide the primary means by
which we establish the properties of $V_T$ that are needed, such as the
\emph{balancing property} (Proposition \ref{balanced}). This latter
property is somewhat awkward to explain, but it concerns the distribution of
generators in words representing the trivial element. The property has
the most force when the generators all lie in peripheral subgroups of
$V_T$. It implies, for instance, that an element of a peripheral subgroup
cannot be expressed efficiently using only generators from \emph{other}
peripheral subgroups; see Remarks \ref{balancedremarks}. The balancing
property plays a key role in several parts of the paper, most notably in
the distortion bound of Proposition \ref{distortion}, and also in the
area bound for $V_T$ given in Proposition \ref{areainVT}. 

\subsection*{Least-area diagrams} In sections \ref{sec:canonical} and
\ref{sec:lower} we define families of van Kampen diagrams in order to
establish the lower bound for the Dehn function of $S_{T,n}$. We use
basic building blocks called \emph{canonical diagrams}, which are defined
for every palindromic word in $\langle x, y\rangle$. These are diagrams
over the vertex group $V_T$. We then build \emph{snowflake diagrams}
over $S_{T,n}$, using canonical diagrams joined along strips dual to the
stable letters of $S_{T,n}$. 

To establish the lower bound, we show that all of these diagrams 
minimize area relative to their boundary words. Proving this requires a
detailed study of $\sigma$--corridors and their intersection properties
for various $\sigma$. 

\subsection*{Edge group distortion} The heart of the computation of the
Dehn function of $S_{T,n}$ lies in Proposition \ref{distortion}, in which
we establish the distortion of the edge groups inside $S_{T,n}$. This
argument requires first some properties of \emph{folded corridors},
analogous to those studied in \cite{BrGr} in the context of
free-by-cyclic groups. These properties are established in the first half
of Section \ref{sec:distortion}, culminating in Lemma
\ref{segmentsnearlyabove}. 

Next we establish the bound on edge group distortion. The proof is an
inductive proof based on Britton's Lemma. It falls into two cases, which
require very different treatments. 
In the first case, the proof is based on a method from \cite{BrBr}. It is
only thanks to the balancing property of $V_T$ that this 
argument can be carried out. 

The second case is where the folded corridors and Lemma
\ref{segmentsnearlyabove} come in. The inductive framework is
based on the nested structure of $r_j$--corridors in a putative van
Kampen diagram. These corridors may appear in two possible orientations,
\emph{forwards} and \emph{backwards}. Forward-facing corridors present
few difficulties and can be handled using the method above based on the
balancing property. If there is a backwards-facing $r_j$--corridor,
then it may introduce undesirable geometric effects that threaten
to spoil the inductive calculation. The argument in this case is to show
that when this occurs, there will be forward facing corridors
just behind the first one, and perfectly matching segments along these
corridors, along which any metric distortion introduced by the first
corridor is exactly undone.

\subsection*{Computing the Dehn function} 
Establishing the upper bound for the Dehn function of $S_{T,n}$
proceeds along similar lines as in \cite{BBFS}. One important step in
this argument is a statement about area in the vertex group. This
occurs in Proposition \ref{areainVT} in the present paper. Due to the
more complicated structure of $V_T$ (being constructed from free
groups), the argument is considerably more involved than the
corresponding result in \cite{BBFS}. It requires many of the tools
developed here, such as the balancing property and corridor schemes.

\subsection*{Acknowledgments}
Noel Brady acknowledges support from the NSF and from NSF award
DMS-0906962. Max Forester acknowledges support from NSF award
DMS-1105765.

 \section{Preliminaries}
 \label{sec:prelim}

In this section we review some basic definitions and properties
concerning Dehn functions, van Kampen diagrams, and words in the free
group. 

\subsection*{Dehn functions}
Let $G= \langle \, A \mid R \, \rangle$ be a finitely presented group and
$w$ a word in the generators $A^{\pm 1}$ representing the trivial
element of $G$. We define the \emph{area} of $w$ to be
\[
\area(w) \ = \ \min \Big\{ N \in \N \ \big| \ \exists \text{ equality } w =
\prod_{i=1}^N u_j r_j u_j^{-1} \text{ freely, where }r_j \in R^{\pm 1} \Big\}.
\]
The \emph{Dehn function} $\delta(x)$ of the finite presentation
$\langle \, A \mid R \, \rangle$ is given by
\[
\delta(x) \ = \ \max \Big\{ \area(w) \ \big| \ w \in \ker(F(A) \to G),
\abs{w} \leq x \Big\}
\]
where $\abs{w}$ denotes the length of the word $w$. 

There is an equivalence relation on functions $f \co \N \to \N$
defined as follows. First, we say that $f \preceq g$ is there is a
constant $C> 0$ such that 
\[
f(x) \ \leq \ C g(Cx) + Cx
\]
for all $x \in \N$. If $f \preceq g$ and $g \preceq f$ then we say
that $f$ and $g$ are \emph{equivalent}, denoted $f \simeq g$. It is
not difficult to show that two finite presentations of the same 
group define equivalent Dehn functions; we therefore speak of ``the''
Dehn function of $G$, which is well defined up to equivalence. 

\begin{remark}\label{equivremark}
In order to show that $f \preceq g$ for non-decreasing functions $f$
and $g$, it is sufficient to prove that $f(n_i) \leq g(n_i)$ for an
unbounded sequence of positive integers $\{n_i\}$ such that the ratios
$n_{i+1}/n_i$ are bounded. For, if $n_{i+1} \leq C n_i$ for all $i$
and $x$ is any integer between $n_i$ and $n_{i+1}$, say, we have
$f(x) \leq f(n_{i+1}) \leq g(n_{i+1}) \leq g(C n_i) \leq g(Cx)$. 
Therefore $f(x) \leq C g(x)$ for all $x$. 
\end{remark}

Let $X$ be a $2$--dimensional cell complex. We call $X$ a
\emph{presentation $2$--complex} if it has one $0$--cell, and every
$2$--cell is attached by a map $f \co S^1 \to X^{(1)}$ which is
\emph{regular} in the following sense: there is a cell structure for
$S^1$ such that the restriction of $f$ to each edge maps monotonically
over a $1$--cell of $X$. 

The presentation $2$--complex of the presentation $\langle \, A \mid R
\, \rangle$ has oriented $1$--cells labeled by the generators in $A$, and a
$2$--cell for each relator $r$ in $R$, attached via a map $S^1 \to
X^{(1)}$ which traverses edges sequentially, following the word $r$. 

Given a presentation $2$--complex $X$, one then has the notion of a
\emph{van Kampen diagram} over $X$. Briefly, a van Kampen diagram for
the word $w$ is a contractible, planar $2$--complex with edges labeled
by generators, with each $2$--cell boundary word equal to a relator, with
outer boundary word $w$. The \emph{area} of
the diagram is the number of its $2$--cells. It is a standard fact that
$\area(w)$ as defined above can be interpreted as the minimal
area of a van Kampen diagram over $X$ for $w$. See \cite{Bridson:word}
for details on this interpretation of $\area(w)$. We refer to
\cite{Bridson:word} for background on Dehn functions generally, and also
to \cite{BH} for background on CAT(0) spaces. 

\subsection*{Words and automorphisms}
A word $w(x,y)$ is \emph{palindromic} if $w(x^{-1},y^{-1}) =
w(x,y)^{-1}$ as words in the free group. An automorphism $\phi \co
\langle x, y\rangle \to \langle x, y \rangle$ is called
\emph{palindromic} if it takes palindromic words to palindromic
words. Note that $\phi$ will be palindromic if the two words $\phi(x)$,
$\phi(y)$ are palindromic. If $\phi$ is palindromic, so is $\phi^n$ for
any $n \geq 1$. 

A word $w(x,y)$ is \emph{positive} if it does not contain occurrences of
$x^{-1}$ or $y^{-1}$. It is \emph{negative} if it does not contain $x$ or
$y$. It is \emph{monotone} if it is positive or negative. Note that
monotone words are reduced, and if $w$ is monotone then so is
$w^{-1}$. If an automorphism $\phi$ takes $x$ and $y$ to monotone words
of the same kind, then it takes all monotone words to monotone words. We
call $\phi$ \emph{monotone} if it has this property. The same will then
be true of $\phi^n$ for any $n \geq 1$. (The inverse $\phi^{-1}$ will
typically \emph{not} be monotone.)

\section{Group-theoretic constructions}\label{sec:groups}

In this section we begin by constructing groups $V_T$ and $W_T$,
which will then serve as vertex groups in graph of groups 
decompositions defining the snowflake group $S_{T,n}$ and the CAT(0)
group $G_{T,n}$. 

Our constructions make use of a free group automorphism $\phi \co
\langle x, y \rangle \to \langle x, y \rangle$ which is both
palindromic and monotone. For concreteness, we define $\phi$ to be the
automorphism given by $\phi(x) = xyx$, $\phi(y) = x$. In most of what
follows, only the palindromic and monotone properties of $\phi$ are
relevant. However, in Section \ref{sec:cat0} we define the CAT(0)
structure for $G_{T,n}$ based on the knowledge that $\phi$ is the
automorphism just defined. This section is the only place where
explicit knowledge of $\phi$ is used. In particular, all results
concerning the groups $S_{T,n}$ are valid for any palindromic, monotone
$\phi$. 

Let $\lambda$ be the exponential growth rate of $\phi$. That is,
$\lambda$ is the Perron-Frobenius eigenvalue of the transition matrix
$M_{\phi} = \begin{pmatrix} \, \abs{\phi(x)}_x & \abs{\phi(y)}_x \\
\abs{\phi(x)}_y & \abs{\phi(y)}_y \, \end{pmatrix}$. 
From the beginning, we will fix an integer $n \geq 1$ and work with the
automorphism $\phi^n$. 

\subsection*{The groups $F$ and $G$}
In this paper $F$ will always denote the free group of rank two,
$\langle x, y\rangle$. We define $G$ to be the free-by-cyclic group $F
\rtimes_{\phi^n} \langle t \rangle$. That is,
\[ G \ = \  \langle \, x, y, t \mid t x t^{-1} = \phi^n(x), \ t y t^{-1} =
\phi^n(y) \, \rangle.\]
One verifies easily that because $\phi^n$ is palindromic, there is an
involution $\tau_G \co G \to G$ defined by $\tau_G(x) =
\overline{x}$, $\tau_G(y) = \overline{y}$, and $\tau_G(t) = t$ (bar
denotes inverse). Similarly, define the involution 
$\tau_F \co F \to F$ by $\tau_F(x) = \overline{x}$, $\tau_F(y) =
\overline{y}$. We may refer to either involution simply as $\tau$. 

The groups we construct will contain many copies of $F$ and $G$. The
different copies will be indexed along with their generators:
$F_{i} = \langle x_i, y_i \rangle$ and $G_i = \langle x_i, y_i \rangle
\rtimes _{\phi^n} \langle t_i \rangle$.

\subsection*{The groups $V$ and $W$} 
These groups are defined as follows:
\begin{align*}
V \ &= \ F_0 \times F_1 \times F_2, \\
W \ &= \ G_0 \times G_1 \times G_2. 
\end{align*}
Thus $V$ has generators $x_0, y_0, x_1, y_1, x_2, y_2$ and $W$
contains $V$ along with the additional generators $t_0, t_1,
t_2$. 

Define the following subgroups (indices mod 3): 
\begin{align*}
A_i \ &= \ \langle \overline{x}_i x_{i+1}, \overline{y}_i y_{i+1} \rangle
\ < \ F_i \times F_{i+1} \ < \ V, \\
B_i \ &= \ \langle \overline{x}_i x_{i+1}, \overline{y}_i y_{i+1},
t_i t_{i+1} \rangle \ < \ G_i \times G_{i+1} \ < \ W. 
\end{align*}
There are injective homomorphisms $F \to F_i \times F_{i+1}$ and $G \to
G_i \times G_{i+1}$ given by $\tau_F \times \id$ and $\tau_G \times
\id$ respectively, taking $x$ to $\overline{x}_i x_{i+1}$, $y$ to
$\overline{y}_i y_{i+1}$, and  $t$ to $t_i t_{i+1}$. The subgroups $A_i$
and $B_i$ are the images of these homomorphisms, and therefore are
isomorphic to $F$ and $G$, with the generators listed above corresponding
to the standard generators $x, y, t$. We name these generators as
follows: 
\[a_i = \overline{x}_i x_{i+1}, \ \ b_i = \overline{y}_i y_{i+1}, \ \ 
{c}_i = t_i t_{i+1}. \] 
Thus, $A_i = \langle a_i, b_i \rangle$ and $B_i = \langle a_i, b_i
\rangle \rtimes_{\phi^n} \langle c_i\rangle$. 
These subgroups will be called \emph{peripheral subgroups}. 

\subsection*{The groups $V_T$ and $W_T$} These groups
will be obtained by amalgamating copies of $V$ (respectively, $W$)
together along peripheral subgroups. 

Let $T$ be a finite tree of valence at most $3$. Choose a copy of $V$ for
each vertex of $T$, and assign  
distinct peripheral subgroups of $V$ to each of the outgoing edges at
that vertex. Then, for each edge in $T$, amalgamate the associated
peripheral subgroups of the two copies of $V$ via the isomorphism
$\tau$ (relative to their standard generating sets).\footnote{There is
  no need to specify a direction because $\tau$ is an involution.}
The resulting group is $V_T$. 
In the case where $T$ is a single vertex, $V_T$ is just $V$. 

The group $W_T$ is defined by the same procedure with vertex groups $W$
instead of $V$. Again, if $T$ has one vertex, then $W_T = W$. 

In order to have consistent notation, assign non-overlapping triples
of indices $(0,1,2)$, $(3, 4, 5)$, etc. to the vertices of $T$, and use
these in place of $(0, 1, 2)$ in the definitions of $V$ and $W$
above. For example, in the case of $W_T$, if a vertex has triple $(3,
4, 5)$, then the vertex group is ${G}_3 \times {G}_4 \times {G}_5$
with peripheral subgroups labeled ${B}_3$, ${B}_4$, and ${B}_5$. The
subgroup ${B}_5$ has standard generators $a_5 = \overline{x}_5 x_3$,
$b_5 = \overline{y}_5 y_3$, and ${c}_5 = t_5 t_3$. 

With this notation, if an edge $e$ of $T$ is assigned the peripheral
subgroups $B_i$ and $B_j$ in its neighboring vertex groups,
then amalgamating along $e$ adds the relations $a_i =
\overline{a}_j$, $b_i = \overline{b}_j$, and ${c}_i = {c}_j$. 

\begin{figure}[ht]
\labellist
\hair 2pt
\small
\pinlabel* {$\textcolor{blue}{0}$} at 12.5 14
\pinlabel* {$\textcolor{blue}{1}$} at 27 14
\pinlabel* {$\textcolor{blue}{2}$} at 19.5 28

\pinlabel* {$\textcolor{blue}{3}$} at 39 20.5
\pinlabel* {$\textcolor{blue}{4}$} at 46.5 33.5
\pinlabel* {$\textcolor{blue}{5}$} at 31.5 33.5

\pinlabel* {$\textcolor{blue}{6}$} at 46.5 47.5
\pinlabel* {$\textcolor{blue}{7}$} at 39 61
\pinlabel* {$\textcolor{blue}{8}$} at 31.5 47.5

\pinlabel* {$\textcolor{blue}{9}$} at 51.5 14
\pinlabel* {$\textcolor{blue}{10}$} at 65 14
\pinlabel* {$\textcolor{blue}{11}$} at 58.5 27

\pinlabel* {$\textcolor{blue}{12}$} at 71  33.5
\pinlabel* {$\textcolor{blue}{13}$} at 78 20.5
\pinlabel* {$\textcolor{blue}{14}$} at 84.5 33.5

\pinlabel* {$\textcolor{blue}{15}$} at 85 47.5
\pinlabel* {$\textcolor{blue}{16}$} at 78 61
\pinlabel* {$\textcolor{blue}{17}$} at 71 47.5

\normalsize
\pinlabel {$A_0$} [tl] at 19 7.5
\pinlabel {$A_9$} [tl] at 58.5 7.5
\pinlabel {$A_{13}$} [Bl] at 86 20
\pinlabel {$A_{15}$} [Bl] at 88.5 55
\pinlabel {$A_{16}$} [Br] at 71 62
\pinlabel {$A_6$} [Bl] at 46 60
\pinlabel {$A_7$} [Br] at 28.5 56
\pinlabel {$A_2$} [Br] at 9 22

\pinlabel {${T}$} at 174 30
\endlabellist
\includegraphics[width=4.5in]{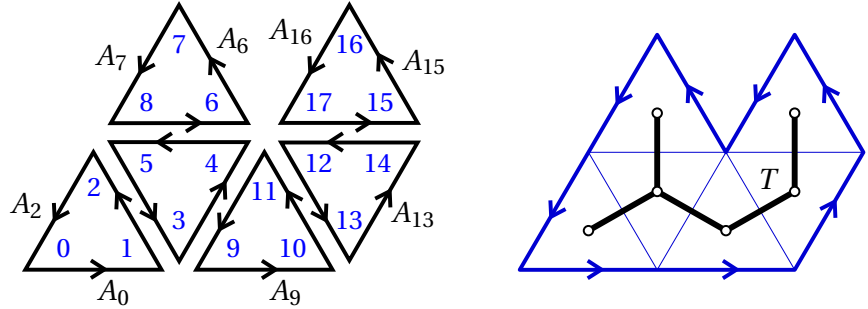}
\caption{The diagram $D$ on the left, and the triangulated
  $(\abs{T}+2)$--gon on the right, with dual graph $T$. Each triangle
  with corners labeled $i, j, k$ represents the vertex group $F_i \times
  F_j \times F_k$ or $G_i \times G_j \times G_k$. The side pairings
  depict the amalgamations $A_1 =_{\tau} A_5$, $A_3 =_{\tau}
  A_{11}$, $A_4 =_{\tau} A_8$, $A_{10} =_{\tau} A_{12}$, and $A_{14}
  =_{\tau} A_{17}$ (in the case of $V_T$).}\label{fig:D} 
\end{figure}
It may be helpful to consider the diagram $D$ as shown in Figure
\ref{fig:D}. It 
has a triangle for each vertex of $T$, with corners corresponding to the
factors $F_i$ or $G_i$ of the vertex group there. The edges correspond to
peripheral subgroups, and the edge-pairings between triangles correspond
to amalgamations. The triangles assemble into a triangulated
$(\abs{T}+2)$--gon with dual graph $T$. (Here, $\abs{T}$ denotes the
number of vertices of $T$.) 

The orientations on edges indicate the standard generating sets of the
peripheral subgroups, relative to the indexing of the groups $F_i$ or
$G_i$. The orientation-reversing nature of the side pairings reflects the
fact that the amalgamations are performed using $\tau$. 

The \emph{peripheral subgroups} of $V_T$ (or of $W_T$) are defined to be
the remaining peripheral subgroups of the vertex groups that were not
assigned to edges of $T$. In terms of the diagram $D$, these are the
peripheral subgroups corresponding to the edges forming the boundary
$(\abs{T}+2)$--gon. 

Next we re-index the peripheral subgroups. Note that the boundary edges
along $D$ are coherently oriented. Start with one and let $\nu_0$ be the
index of the corresponding peripheral subgroup. Following the
orientation, let $\nu_1$ be the index of the next edge along $\partial
D$. Repeat in this way and define the indices $\nu_0 \dotsc, \nu_m$,
allowing us to refer to the peripheral subgroups as $A_{\nu_0}, \dotsc,
A_{\nu_m}$ (or $B_{\nu_0}, \dotsc, B_{\nu_m}$). Note, $m = \abs{T} + 1$.  

\subsection*{The groups $S_{T,n}$ and $G_{T,n}$} 
Fix a tree $T$ as above and let $m = \abs{T} + 1$. 
The group $S_{T,n}$ is defined to be a multiple HNN extension over $V_T$
with stable letters $r_1, \dotsc, r_m$, where $r_i$ conjugates the 
peripheral subgroup $A_{\nu_0}$ to $A_{\nu_i}$ via the automorphism
$\phi^n$. That is,
\begin{align}\label{smn}
S_{T,n} = \langle \, V_T, r_1, \dotsc, r_m \mid \ &r_i a_{\nu_0} r_i^{-1} =
\phi^n(a_{\nu_i}), \notag \\ &r_i b_{\nu_0} r_i^{-1} = \phi^n(b_{\nu_i})  
\text{ for each } i \, \rangle.
\end{align} 
Thus $S_{T,n}$ is the fundamental group of a graph of groups whose
underlying graph is the $m$--rose (having one vertex and $m$ loops). The
vertex group is $V_T$ and the edge groups are all $F$. 

We define $G_{T,n}$ in a similar manner, but \emph{without twisting}. It
is a multiple HNN extension over $W_T$ with stable letters $s_1, \dotsc,
s_m$, where $s_i$ conjugates $B_{\nu_0}$ to $B_{\nu_i}$ via the
identity map:
\begin{align}\label{gmn1}
G_{T,n} = \langle \, W_T, s_1, \dotsc, s_m \mid \ &s_i a_{\nu_0} s_i^{-1} =
a_{\nu_i}, \notag \\ &s_i b_{\nu_0} s_i^{-1} = b_{\nu_i}, \, s_i {c}_{\nu_0}
s_i^{-1} = {c}_{\nu_i} \text{ for each } i \, \rangle.
\end{align}
Again, $G_{T,n}$ is the fundamental group of a graph of groups over the
$m$--rose. The vertex group is $W_T$ and the edge groups are all $G$.

\section{The CAT(0) structure} \label{sec:cat0}

In this section we build the CAT(0) structure for the group
$G_{T,n}$. 

\subsection*{The space $Z$} Recall that we have chosen a specific 
monotone palindromic automorphism $\phi \co F \to F$ given by $\phi(x) =
xyx$, $\phi(y) = x$. Let 
\begin{equation}\label{G0presentation}
G_0 \ = \ \langle x, y\rangle \rtimes_{\phi} \langle {t_0} \rangle \ 
= \ \langle \, x, y, {t_0} \mid {t_0}x{t_0}^{-1} = xyx, \
{t_0}y{t_0}^{-1} = x \, \rangle. 
\end{equation}
T. Brady \cite{TBr} has constructed a piecewise Euclidean, locally CAT(0)
$2$--complex $Z_0$ with fundamental group $G_0$. This $2$--complex has
two vertices, four edges, and two $2$--cells, consisting of a Euclidean
octagon and quadrilateral as shown in Figure \ref{fig:Z0complex}. The
angle $\alpha \in (0,\pi)$ is a free parameter, and the rest of the
geometry (up to scaling) is then determined. It is easy to check that
both vertices satisfy the link condition, making $Z_0$ locally CAT(0). 
\begin{figure}[ht]
\labellist
\hair 2pt
\small
\pinlabel {$\tau_{Z_0}$} [r] at 31 6
\pinlabel {$\tau_{Z_0}$} [r] at 130 38

\pinlabel {$\textcolor[HTML]{37bc62}{\alpha}$} [Bl] at 46 122
\pinlabel {$\textcolor[HTML]{37bc62}{\alpha}$} [Bl] at 144 86
\pinlabel {$\textcolor[HTML]{37bc62}{\alpha}$} [tl] at 145 51.5
\pinlabel {$\textcolor[HTML]{37bc62}{\alpha}$} [tl] at 47 18
\pinlabel {$\textcolor[HTML]{37bc62}{\pi}$} [l] at 16.5 68
\pinlabel {$\textcolor[HTML]{37bc62}{\pi}$} [r] at 65 68

\pinlabel {$z_0$} [r] at 6 31
\pinlabel {$z_0$} [l] at 74.5 31
\pinlabel {$z_0$} [r] at 6 104
\pinlabel {$z_0$} [l] at 74.5 104

\pinlabel {$\textcolor[HTML]{0000dd}{x}$} [tl] at 43.5 101.5
\pinlabel {$\textcolor[HTML]{0000dd}{y}$} [Bl] at 43 38.5
\pinlabel {$\textcolor[HTML]{0000dd}{t_0}$} [l] at 141.5 68
\endlabellist
\includegraphics[width=2.5in]{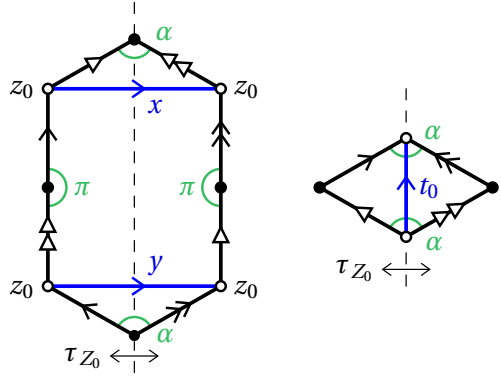}
\caption{The piecewise Euclidean $2$--complex $Z_0$ and its isometric
  involution $\tau_{Z_0}$.}\label{fig:Z0complex} 
\end{figure}

The figure also shows three arcs crossing the interiors of the
$2$--cells, representing the elements $x$, $y$, and ${t_0}$ in $\pi_1(Z_0,
z_0)$; we leave it to the reader to verify that $\pi_1(Z_0,z_0)$ indeed
has the presentation \eqref{G0presentation} relative to these
generators. 

Reflection of each $2$--cell across the vertical dotted lines in Figure
\ref{fig:Z0complex} respects the edge identifications, and induces an
isometric involution $\tau_{Z_0} \co Z_0 \to Z_0$. The induced
homomorphism $\tau_{G_0} \co G_0 \to G_0$ is given by $\tau_{G_0}(x) =
\overline{x}$, $\tau_{G_0}(y) = \overline{y}$, and $\tau_{G_0}({t_0}) =
{t_0}$. 

Let $t = (t_0)^n$ and note that the index-$n$ subgroup $\langle x,
y, t\rangle \triangleleft G_0$ is the group $G$ defined earlier. The
corresponding covering space $Z$ of $Z_0$ has a locally CAT(0) structure
made from $n$ octagons and $n$ quadrilaterals. The involution
$\tau_{Z_0}$ lifts to an isometric involution $\tau_Z \co Z \to Z$, with
induced homomorphism $\tau_G$. The case $n=3$ is shown in Figure
\ref{fig:Zcomplex}. 
\begin{figure}[ht]
\labellist
\hair 2pt
\small
\pinlabel {$\tau_Z$} [r] at 11 75
\pinlabel {$\tau_Z$} [r] at 65 93
\pinlabel {$\tau_Z$} [r] at 119 40
\pinlabel {$\tau_Z$} [r] at 173 57
\pinlabel {$\tau_Z$} [r] at 227 4
\pinlabel {$\tau_Z$} [r] at 281 21
\endlabellist
\includegraphics[width=4.5in]{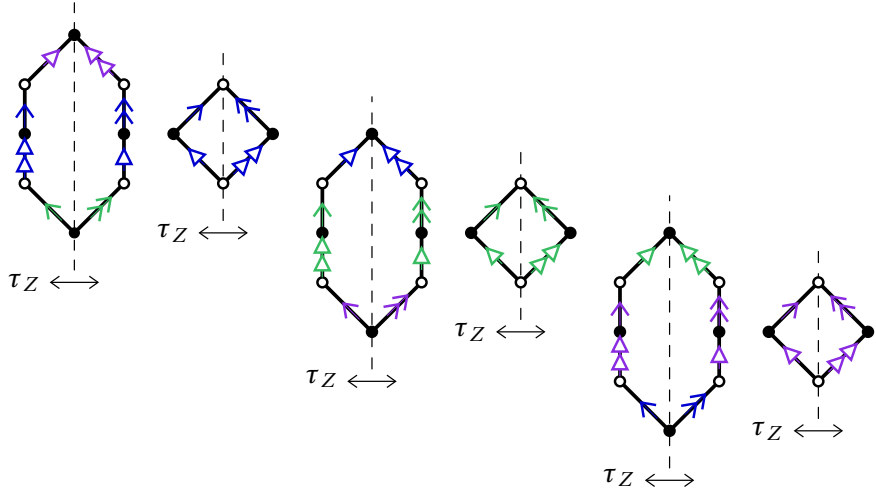}
\caption{The locally CAT(0) $2$--complex $Z$ and its involution
  $\tau_Z$, when $n=3$.}\label{fig:Zcomplex} 
\end{figure}

To summarize, we now have a locally CAT(0) space $Z$ with fundamental
group $G$, and an isometric involution $\tau_Z \co Z \to Z$ whose induced
homomorphism is given by the involution $\tau_G$. 

\subsection*{The space $K_T$} 
We shall need the following ``gluing with a tube'' result. 

\begin{proposition}[\cite{BH}, II.11.13]\label{gluing} 
Let $X$ and $A$ be locally CAT(0) metric spaces. If $A$ is compact and
$\phi, \psi \co A \to X$ are locally isometric immersions, then the
quotient of $X \sqcup (A \times [0,1])$ by the equivalence relation
generated by $(a,0) \sim \phi(a); (a,1) \sim \psi(a)$, $\forall a \in
A$ is locally CAT(0). 
\end{proposition}

Let $Z_i$ and $Z_j$ be copies of $Z$ with fundamental groups $G_i$ and
$G_j$ repsectively. Let $Z_i \times Z_j$ be given the product metric. 
Define the map $f_{i,j} \co Z \to Z_i \times Z_j$ by $f_{i,j}(p) =
(\tau_z(p), p)$. Note that the induced homomorphism $(f_{i,j})_* \co G
\to G_i \times G_j$ is given by $x \mapsto \overline{x}_i x_j$,
$y \mapsto \overline{y}_i y_j$, $t \mapsto t_i t_j$. Metrically,
$f_{i,j}$ behaves as follows: 
\begin{align*}
d(f_{i,j}(p),f_{i,j}(q)) \ &= \ d((\tau_Z(p),p), (\tau_Z(q),q)) \\
&= \ (d(\tau_Z(p),\tau_Z(q))^2 + d(p,q)^2)^{1/2} \\
&= \ (2d(p,q)^2)^{1/2} \ = \ \sqrt{2}d(p,q). 
\end{align*}
Hence $f_{i,j}$ is an isometric embedding of the scaled metric space
$(\sqrt{2})Z$ into $Z_i \times Z_j$. 

Now we define the locally CAT(0) space $K_{T}$ with fundamental group
$W_T$. Let 
\[K = Z_0 \times Z_1 \times Z_2\] 
where each $Z_i$ is a copy of $Z$. Thus, $K$ is locally CAT(0) and 
has fundamental group $W = G_0 \times G_1 \times G_2$. Fix a basepoint 
$v\in Z$ and let $v_i \in Z_i$ be the corresponding points. Each product
space $Z_i \times Z_{i+1}$ isometrically embeds into $K$ using the
basepoint $v_{i-1}$ as the missing coordinate (indices mod~$3$). Define
the \emph{peripheral subspace} $L_i$ to be the image of the map 
\begin{equation}\label{inclusion}
Z \ \overset{f_{i, i+1}}{\longrightarrow} \ Z_i \times Z_{i+1} \
\hookrightarrow \ K.
\end{equation} 
Note that $L_i$ has fundamental group $B_i < W$ and the induced map
$\pi_1(Z, v) \to W$ is the standard inclusion map $x \mapsto a_i$, $y
\mapsto b_i$, $t \mapsto c_i$ with image $B_i$. 

The space $K_T$ is formed from copies of $K$ in the same way that $W_T$
is built from copies of $W$. Take a copy of $K$ for each vertex of $T$, 
with all indices re-named to agree with the triple of indices assigned to
that vertex. Whenever $B_i$ and $B_j$ were amalgamated in $W_T$, glue the 
ends of a \emph{tube} $(\sqrt{2})Z \times [0,1]$ to the peripheral
subspaces $L_i$ and $L_j$, using the isometric embedding
\eqref{inclusion} from $(\sqrt{2})Z \times \{0\}$ to the copy of $K$
containing $L_i$, and using a similar isometric embedding
\[Z \overset{\tau_Z}{\longrightarrow} \ Z \ \overset{f_{j,
 j+1}}{\longrightarrow} \ Z_j \times Z_{j+1} \ \hookrightarrow \ K \]
from $(\sqrt{2})Z \times \{1\}$ to the copy of $K$ containing 
$L_j$. The involution $\tau_Z$ is being used to obtain the correct
identification between the subgroups $B_i$ and $B_j$. The resulting space
$K_T$ has fundamental group $W_T$, and is locally CAT(0) by Proposition
\ref{gluing}. In particular, $W_T$ is CAT(0). 

\begin{remark}\label{VTisCAT0}
The reasoning above also shows that $V_T$ is CAT(0). One simply
re-defines $Z$ to be the space $S^1 \vee S^1$ with any path metric (which
will be locally CAT(0)). There is an obvious isometric involution
$\tau_Z$ which reverses the direction of each loop in $Z$, and induces
$\tau_F \co F \to F$. The rest is entirely similar. 
\end{remark}

\subsection*{The space $K_{T,n}$}
Inside $K_T$ there are peripheral subspaces $L_{\nu_0}, \dotsc,
L_{\nu_m}$. For each $j = 1, \dotsc, m$ glue the ends of a tube
$(\sqrt{2})Z \times [0,1]$ to $L_{\nu_0}$ and $L_{\nu_j}$ using the
isometric embeddings \eqref{inclusion} from $(\sqrt{2})Z \times \{0\}$
and $(\sqrt{2})Z \times \{1\}$ to the appropriate copies of $K$ in
$K_T$. The resulting space $K_{T,n}$ is the total space of a graph of
spaces corresponding to the description \eqref{gmn1} of $G_{T,n}$ as the
fundamental group of a graph of groups. In particular, $K_{T,n}$ has
fundamental group $G_{T,n}$. It is locally CAT(0) by Proposition
\ref{gluing}. Thus we have proved: 

\begin{theorem}\label{GTnisCAT0}
$G_{T,n}$ is CAT(0). \qed
\end{theorem}

\section{Embedding results} \label{sec:embedding}

In this section we define the embedding $S_{T,n} \to G_{T,n}$
and prove that it is injective. The map is defined step by step,
following the constructions defining $S_{T,n}$ and $G_{T,n}$. In several
places we use the following lemma to establish injectivity. It is a
special case of a basic result of Bass \cite{Bass}, reformulated
slightly. 

\begin{lemma}[Injectivity for graphs of groups] \label{basslemma}
Suppose $\mathscr{A}$ and $\mathscr{B}$ are graphs of groups such that
the underlying graph $\Gamma_{\mathscr{A}}$ of $\mathscr{A}$ is a
subgraph of the underlying graph of $\mathscr{B}$. Let $A$ and $B$ be
their respective fundamental groups. Suppose that there are injective
homomorphisms $\psi_e \co A_e \to B_e$ and $\psi_v \co A_v \to B_v$
between edge and vertex groups, for all edges $e$ and vertices $v$ in
$\Gamma_{\mathscr{A}}$, which are compatible with the edge-inclusion
maps. 

(That is, whenever $e$ has initial vertex $v$, the diagram 
\[\begin{CD}
A_e @>>> A_{v}\\
@V{\psi_e}VV        @V{\psi_{v}}VV\\
B_e @>>> B_{v}
\end{CD}\]
commutes.) 

If $\psi_e(A_e) = \psi_{v} (A_{v}) \cap B_e$ whenever $e$ has initial
vertex $v$, then the induced homomorphism $\psi \co A \to B$ is
injective. 
\end{lemma}

\begin{remark}
Given the initial assumptions, it is always true that $\psi_e(A_e)
\subset (\psi_v(A_v) \cap B_e)$. In practice one only needs to verify
that $\psi_e(A_e)$ contains $\psi_v(A_v) \cap B_e$. 
\end{remark}

\begin{proof}
The homomorphisms $\psi_e$, $\psi_v$ combine to give a morphism of graphs
of groups in the sense of Bass \cite{Bass}. According to Proposition 2.7 of
\cite{Bass}, $\psi\co A \to B$ will be injective if, whenever $e$ has
initial vertex $v$, the function $A_v /A_e \to B_v / B_e$ induced by
$\psi_v$ is injective. 

To prove that the latter statement holds, suppose that the cosets
$\psi_v(a_1)B_e$ and $\psi_v(a_2)B_e$ are equal 
for some $a_1, a_2 \in A_v$. Then $\psi_v(a_1a_2^{-1}) \in B_e$, and
hence (by the main assumption) $\psi_v(a_1a_2^{-1}) \in
\psi_e(A_e)$. Since $\psi_e$ is injective (and agrees with $\psi_v$),
$a_1 a_2^{-1} \in A_e$, and therefore $a_1 A_e = a_2 A_e$. 
\end{proof}

\begin{lemma}\label{intersect}
Let $\iota \co V \to W$ be inclusion. Then $\iota(A_i) = \iota(V) \cap
B_i$ for each $i$. 
\end{lemma}

\begin{proof}
Without loss of generality let $i = 0$. Note that $\iota(A_0)$ and
$\iota(V) \cap {B}_0$ are both contained in the subgroup ${G}_0 \times
{G}_1$, so it suffices to show that $\iota(A_0) = \iota(F_0 \times F_1)
\cap {B}_0$ in ${G}_0 \times {G}_1$. One direction, $\iota(A_0)
\subset \iota(F_0 \times F_1) \cap {B}_0$, is obvious. 

For the other direction, consider an element of
$\iota(F_{0}\times F_{1}) \cap B_0$. It can be expressed as a word
$w(\overline{x}_0 x_1, \overline{y}_0 y_1, t_0 t_1)$, which equals
$w(\overline{x}_0, \overline{y}_0, t_0)w(x_1, y_1, t_1)$ in
$G_{0}\times G_{1}$. Being in $\iota(F_{0}\times F_{1})$ it also
has an expression of the form $u(\overline{x}_0, \overline{y}_0)
v(x_1, y_1)$ where $u$ and $v$ are reduced words in the free
group. Projecting onto the second factor of $G_{0} \times G_{1}$, one
obtains the equation in $G$:
\[ w(x, y, t) = v(x, y).\] 
Considering $G$ as an HNN extension with vertex group $F$ and stable
letter $t$, the right hand side is a word in normal form (that is, a word
of length $1$ consisting of an element of $F$), and therefore gives the
(unique) normal form representative for the element $w$. Similarly,
considering $\{\overline{x}_0, \overline{y}_0\}$ as a basis for
$F_{0}$, projecting onto the first factor gives the equation in $G$ 
\[w(x, y, t) = u(x,y)\]
and hence $u(x,y)$ is also the normal form for $w$. We conclude that
$u$ and $v$ represent the same element of $F$. Since both words are reduced,
they are equal as words and so $u(\overline{x}_0,
\overline{y}_0) v(x_1, y_1) = 
u(\overline{x}_0, \overline{y}_0) u(x_1, y_1) = 
u(\overline{x}_0x_1, \overline{y}_0 y_1) \in \iota(A_0)$. 
\end{proof}

\begin{proposition}
The inclusion maps $F_i \hookrightarrow G_i$ induce an injective
homomorphism $V_T \to W_T$. 
\end{proposition}

Henceforth we will regard $V_T$ as a subgroup of $W_T$. 

\begin{proof}
We will use Lemma \ref{basslemma} since $V_T$ and $W_T$ are both
fundamental groups of graphs of groups with underlying graph $T$. 

To elaborate on the graph of groups structures of $V_T$ and $W_T$, fix an
orientation of each edge of $T$ and use these to specify the
edge-inclusion maps as follows. For $V_T$, each edge group is $F$ and
the two neighboring vertex groups are isomorphic to $V$. For the
\emph{initial} vertex, the inclusion map $F \to F_i \times F_{i+1}$
is $\tau_F \times \id$, and for the \emph{terminal} vertex, the
inclusion map $F \to F_j \times F_{j+1}$ is $\id \times \tau_F$. In the
case of $W_T$, the same convention is used: inclusion maps $G \to G_i
\times G_{i+1}$ are $\tau_G \times \id$ for initial vertices and
$\id\times \tau_G$ for terminal vertices. 

The inclusion maps $F_i \to G_i$ induce inclusions between corresponding
vertex groups $V \to W$. The compatibility diagrams become
\[\begin{CD}
F @>{\tau_F \times \id}>> F_i \times F_{i+1} @>>> V \\
@VVV     &  & @VVV\\
G @>{\tau_G \times \id}>> G_i \times G_{i+1} @>>> W
\end{CD} \ \ 
\text{or} \ \ 
\begin{CD}
F @>{\id \times \tau_F}>> F_i \times F_{i+1} @>>> V \\
@VVV     &  & @VVV\\
G @>{\id \times \tau_G}>> G_i \times G_{i+1} @>>> W 
\end{CD}
\]
and these clearly commute (all unnamed maps are inclusion). Thus there is
an induced homomorphism $V_T \to W_T$. The last condition needed by
Lemma \ref{basslemma} is provided by Lemma \ref{intersect}, and so
we conclude from \ref{basslemma} that $V_T \to W_T$ is injective. 
\end{proof}

\subsection*{Change of coordinates in $G_{T,n}$} 
We plan to use Lemma \ref{basslemma} to embed $S_{T,n}$ into $G_{T,n}$,
but first we must modify the graph of groups structure of $G_{T,n}$. The
modification amounts to a change in the choice of stable letters in the
multiple HNN extension description of $G_{T,n}$. Alternatively, it can
be seen as an application of Tietze transformations. 

Indeed, one can start with the presentation \eqref{gmn1} defining
$G_{T,n}$, add new generators $u_1, \dotsc, u_m$ and relations $u_i =
{c}_{\nu_i} s_i$, replace occurrences of $s_i$ with ${c}_{\nu_i}^{-1}
u_i$, and delete the generators $s_i$. The relation $s_i
a_{\nu_0}s_i^{-1} = a_{\nu_i}$ becomes ${c}_{\nu_i}^{-1} u_i a_{\nu_0}
u_i^{-1} {c}_{\nu_i} = a_{\nu_i}$, or equivalently, $u_i a_{\nu_0}
u_i^{-1} = {c}_{\nu_i} a_{\nu_i} {c}_{\nu_i}^{-1} =
\phi^{n}(a_{\nu_i})$. Similarly, the relation $s_i b_{\nu_0}s_i^{-1} =
b_{\nu_i}$ becomes $u_i b_{\nu_0} u_i^{-1} = \phi^n(b_{\nu_i})$ and the
relation $s_i {c}_{\nu_0} s_i^{-1} = {c}_{\nu_i}$ becomes $u_i
{c}_{\nu_0} u_i^{-1} = {c}_{\nu_i}$. Thus one obtains the new
presentation 
\begin{align}\label{gmn2}
G_{T,n} = \langle \, W_T, u_1, \dotsc, u_m \mid \ &u_i a_{\nu_0} u_i^{-1} =
\phi^n(a_{\nu_i}), \,\notag \\  &u_i b_{\nu_0} u_i^{-1} =
\phi^n(b_{\nu_i}), \, u_i {c}_{\nu_0} u_i^{-1} = {c}_{\nu_i} \text{ for
  each } i \, \rangle. 
\end{align}
This is evidently the presentation arising from a new description of
$G_{T,n}$ as a multiple HNN extension of $W_T$ with stable letters $u_1,
\dotsc, u_m$, where $u_i$ conjugates $B_{\nu_0}$ to $B_{\nu_i}$
via $\phi^n \times \id$. 

\begin{theorem}\label{injectivity}
The homomorphism $S_{T,n} \to G_{T,n}$ induced by the inclusion $V_T
\hookrightarrow W_T$ and the assigment $r_i \mapsto u_i$ is injective. 
\end{theorem}

\begin{proof}
First, given the presentations \eqref{smn} and \eqref{gmn2}, it is clear
that the given assignment defines a homomorphism $S_{T,n} \to
G_{T,n}$. Furthermore, this is the homomorphism induced by the injective
maps on vertex and edge groups: $V_T \hookrightarrow W_T$ in the case of
the vertex, and $F \hookrightarrow G$ for each edge of the $m$--rose. The
corresponding compatibility diagrams are 
\[\begin{CD}
F @>{\id}>> A_{\nu_0} @>>> V_T \\
@VVV     &  & @VVV\\
G @>{\id}>> B_{\nu_0} @>>> W_T
\end{CD} \ \ \ \ 
\text{and} \ \ \ \ 
\begin{CD}
F @>{\phi^n}>> A_{\nu_i} @>>> V_T \\
@VVV     &  & @VVV\\
G @>{\phi^n \times \id}>> B_{\nu_i} @>>> W_T 
\end{CD}
\]
where $A_j$ and $B_j$ are canonically identified with $F$ and $G$
via their standard generating sets, and the unnamed maps are inclusion. 
These diagrams commute. 

We have all of the initial hypotheses of Lemma \ref{basslemma}
satisfied. It remains to verify that $A_{\nu_i} = V_T \cap
B_{\nu_i}$ in $W_T$ for $i = 0, \dotsc, m$. Consider the vertex of
$T$ whose triple of indices includes $\nu_i$. Let $V$ and $W$ be the
vertex groups at that vertex (for the graph of groups decompositions of
$V_T$ and $W_T$). Then $A_{\nu_i}$ and $V_T \cap B_{\nu_i}$ are both
contained in $W$. Moreover $W \cap V_T = V$, and so it suffices to show
that $A_{\nu_i} = V \cap B_{\nu_i}$. This holds by Lemma
\ref{intersect}. Hence, by Lemma \ref{basslemma}, the map $S_{T,n} \to
G_{T,n}$ is injective. 
\end{proof}

\section{Corridor schemes and the balancing property of
  $V_T$}\label{sec:corr}  

In this section we develop two key tools which will play an
essential role throughout the rest of the paper. These tools are specific
to the groups $V_T$, and they are the primary means by which we establish
the various properties of $V_T$ that are needed. The first of these is
the notion of $\sigma$--corridors in van Kampen diagrams over $V_T$. The
second is the \emph{balancing property} of $V_T$, given in Proposition
\ref{balanced}. 

In order to discuss $\sigma$--corridors we first define \emph{corridor
schemes}. We will make use of several corridor schemes in this paper,
in addition to $\sigma$--corridors. In this section we also discuss
the standard generating set for $V_T$, and various notions of length
associated with this generating set. 

\subsection*{The $2$--complex $X_T$} 
In order to discuss area in $V_T$ we will work with a specific 
$2$--complex $X_T$ with fundamental group $V_T$. 

The group $V$ has a presentation with generators $x_i, y_i, a_i, b_i$ for
$i = 0, 1, 2$ and eighteen relations (see also Figure \ref{fig:S-sigma}): 
\begin{gather}
a_i = \overline{x}_i x_{i+1}, \ \ a_i = x_{i+1} \overline{x}_i, \ \ b_i =
\overline{y}_i y_{i+1}, \ \ b_i = y_{i+1} \overline{y}_i, \notag \\  
x_i y_{i+1} = y_{i+1} x_i, \ \ x_{i+1} y_i = y_i x_{i+1} 
\ \ (i = 0, 1, 2 \mod 3) \label{trianglerelations}
\end{gather}
We define $X$ to be the presentation $2$--complex for this presentation
of $V$. Thus $X$ has one $0$--cell, twelve labeled, oriented $1$--cells,
twelve triangular $2$--cells, and six quadrilateral $2$--cells. 

For each $i$, the subcomplex $Y_i \subset X$ consisting of the two
$1$--cells labeled $a_i$ and $b_i$ is called a \emph{peripheral
  subspace}. It is homeomorphic to $S^1 \vee S^1$ and has fundamental
group $A_i \subset V$. 

The $2$--complex $X_T$ is formed from copies of $X$ in the same way that
$V_T$ is built from copies of $V$. Take a copy of $X$ for each vertex of
$T$, with edge labels re-indexed according to the triple of indices
assigned to that vertex. Whenever $A_i$ and $A_j$ were amalgamated in
$V_T$, glue the peripheral subspaces $Y_i$ and $Y_j$ via a cellular
homeomorphism which induces $\tau$ between $A_i$ and $A_j$. The resulting
space $X_T$, with fundamental group $V_T$, has a natural cell
structure. The $1$--cells are labeled by the generators $x_i$, $y_i$,
$a_i$, and $b_i$, where in some cases, a $1$--cell labeled $a_i$ or $b_i$
is also labeled $a_j$ or $b_j$ in the opposite direction. The $2$--cells
are the same as those of the copies of $X$, with the same triangular and
quadrilateral boundary relations. 

\subsection*{Area} In order to simplify the area calculations to follow,
we declare each triangular cell of $X_T$ to have area $1$, and each
quadrilateral cell to have area $2$. (Think of it as being made of
two triangles.) 

\subsection*{Corridor schemes}
Let $Z$ be any presentation $2$--complex. A \emph{corridor scheme} for
$Z$ is a subset ${\Ss}$ of the set of edges of $Z$ such that every $2$--cell
of $Z$ has either zero or two occurrences of edges of ${\Ss}$ in its
boundary. Given such an ${\Ss}$, one can then define \emph{corridors} in any
van Kampen diagram over $Z$. Call the $2$--cells having two ${\Ss}$--edges in
their boundaries \emph{corridor cells}. Given a van Kampen diagram
$\Delta$, two corridor cells in $\Delta$ are called \emph{neighbors} if
they meet along an ${\Ss}$--edge in their boundaries. A corridor cell has
zero, one, or two neighbors. A \emph{corridor} is a minimal collection
$C$ of corridor cells in $\Delta$ such that if $c \in C$ then all
neighbors of $c$ are also in $C$. Every corridor cell is contained in a
corridor. 

Corridors come in two types: those in which every corridor cell has
neighbors along both of its ${\Ss}$--edges, called \emph{annulus type}, and
the others, called \emph{band type}. Each band type corridor joins an
${\Ss}$--edge on the boundary of $\Delta$ to another ${\Ss}$--edge on the
boundary of $\Delta$, and contains no other ${\Ss}$--edges on the boundary of
$\Delta$. An annulus type corridor may have $2$--cells meeting the
boundary of $\Delta$, but the ${\Ss}$--edges of such $2$--cells will not be
on the boundary. 

If $C$ is a corridor in $\Delta$, then the subset formed by taking the
union of the interiors of its $2$--cells along with the interiors of its
${\Ss}$--edges is an open set, homeomorphic to a tubular neighborhood of a
properly embedded connected $1$--dimensional submanifold of $\Delta$. 
The $1$--manifold meets each corridor cell in an arc joining
the two ${\Ss}$--edges of the cell. Thus an annulus type corridor
contains an embedded open annulus, and a band type corridor contains an
embedded open band $[0,1] \times (0,1)$ meeting the boundary of the
diagram in its boundary $\{0,1\}\times (0,1)$. 

Corridors have two key properties. First, two corridors in a diagram
$\Delta$ will never have $2$--cells or ${\Ss}$--edges in common. In
particular, corridors cannot \emph{cross}. Second, every ${\Ss}$--edge
appearing on the boundary of $\Delta$ is part of a band type corridor,
unless that edge is not in any $2$--cell of $\Delta$. In particular,
given an ${\Ss}$--edge in the boundary of $\Delta$, if there is a $2$--cell
containing that edge, then one can pass from neighbor to neighbor in the
corridor, until one arrives at a second, uniquely determined, ${\Ss}$--edge
on the boundary of $\Delta$. Also, in the boundary, this pair of
${\Ss}$--edges cannot be linked with another such pair, because corridors do
not cross. 

\subsection*{Orientable corridor schemes} 
A corridor scheme ${\Ss}$ is \emph{orientable} if there is a choice of
orientations of the edges of ${\Ss}$ such that in each corridor cell, the two
${\Ss}$--edges have oppposite orientations relative to the boundary of the
cell. It follows that in any corridor, the transverse orientations of the 
${\Ss}$--edges along the corridor all agree. 

\begin{remark}
A corridor scheme defines a
$1$--dimensional $\Z_2$--valued cellular cocycle in $Z$. (If $Z$
happens to be a simplicial complex, then every $1$--cocycle 
is a corridor scheme.) An orientable corridor scheme defines a
$\Z$--valued $1$--cocycle in $Z$. See Gersten \cite{Ge} for a thorough
study of corridors from the cohomological point of view. 
\end{remark}

\subsection*{${\sigma}$--corridors} 
Recall that $D$ was a diagram made of triangles, with dual graph $T$,
which may be regarded as being embedded as a subspace of a triangulated
$(\abs{T}+2)$--gon. Let $\widehat{T}$ be a tree obtained from $T$ by
joining the midpoint of each boundary edge to the vertex in the
neighboring triangle. Then $\widehat{T}$ has $m+1$ leaves, corresponding
to the peripheral subgroups of $V_T$, and interior vertices all of
valence three, which are the original vertices of $T$. Denote the leaves
of $\widehat{T}$ by $v_{\nu_0}, \dotsc, v_{\nu_m}$, so that $v_{\nu_i}$
corresponds to the peripheral subgroup $A_{\nu_i}$. 

Let ${\sigma}$ be a maximal segment in $\widehat{T}$. Note that ${\sigma}$ is
uniquely determined by its endpoints; choosing ${\sigma}$ amounts to
choosing a pair of peripheral subgroups of $V_T$. 
For each such ${\sigma}$ we will define a corridor scheme ${\Ss}_{{\sigma}}$ for
$X_T$. 

In the $(\abs{T}+2)$--gon, ${\sigma}$ starts on a boundary edge, passes
through a sequence of triangles, and ends on a boundary edge. Its
intersection with each of these triangles is an arc joining two sides. It
separates one corner of the triangle from the other two. If $i$ is the
index of this corner, put the edges of $X_T$ labeled by $x_i$ and $y_i$
into ${\Ss}_{{\sigma}}$. Also, if ${\sigma}$ passes through the side of a
triangle associated with the subgroup $A_j$, put the edges labeled by
$a_j$ and $b_j$ into ${\Ss}_{{\sigma}}$. Do this for each triangle that
intersects ${\sigma}$ to obtain ${\Ss}_{{\sigma}}$. The fact that some edges of
$X_T$ have two labels is not a problem; either both labels or neither
label will be chosen for inclusion in ${\Ss}_{{\sigma}}$. See Figure
\ref{fig:sigma}. 
\begin{figure}[ht]
\labellist
\hair 2pt
\large
\pinlabel {$\textcolor[HTML]{37bc62}{\sigma}$} [tr] at 30 7
\pinlabel {$\textcolor[HTML]{37bc62}{\sigma}$} [Bl] at 82 79

\normalsize
\pinlabel* {$\textcolor{blue}{0}$} at 17 20
\pinlabel* {$\textcolor{blue}{1}$} at 48 20
\pinlabel* {$\textcolor{blue}{2}$} at 32.5 47
\pinlabel* {$\textcolor{blue}{3}$} at 79.5 38
\pinlabel* {$\textcolor{blue}{4}$} at 95 65
\pinlabel* {$\textcolor{blue}{5}$} at 63 65

\pinlabel {$A_1$} [Bl] at 40 52
\pinlabel {$A_5$} [tr] at 73 36
\pinlabel {$A_3$} [tl] at 102 63
\pinlabel {$A_4$} [Br] at 71 78
\pinlabel {$A_2$} [Br] at 9.5 27
\pinlabel {$A_0$} [tl] at 41 9
\endlabellist
\includegraphics{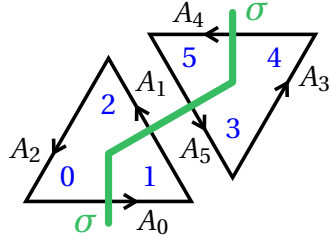}
\caption{This portion of ${\sigma}$ passing through $D$ contributes the
  following edges of $X_T$ to ${\Ss}_{{\sigma}}$: $a_0$, $b_0$, $x_1$, $y_1$,
  $a_1$ ($=a_5$), $b_1$ ($= b_5$), $x_5$, $y_5$, $a_4$,
  $b_4$.}\label{fig:sigma} 
\end{figure}

One verifies easily that ${\Ss}_{{\sigma}}$ is a corridor scheme, by
examining the relations \eqref{trianglerelations} for each triangle of
$D$. See also Figure \ref{fig:S-sigma}. (Because of the two-label
phenomenon, it is important here that ${\sigma}$ is \emph{maximal}.)
Corridors defined by this scheme will be called
\emph{${\sigma}$--corridors}. 

\begin{figure}[ht]
\labellist
\small\hair 2pt
\pinlabel {$x_0$} [b] at 12 100.5
\pinlabel {$x_0$} [b] at 53 100.5
\pinlabel {$x_1$} [b] at 93 100.5
\pinlabel {$x_1$} [b] at 133 100.5
\pinlabel {$x_2$} [b] at 174 100.5
\pinlabel {$x_2$} [b] at 214 100.5

\pinlabel {$y_1$} [r] at 2 90.5
\pinlabel {$y_1$} [l] at 22.5 90.5
\pinlabel {$y_2$} [r] at 42 90.5
\pinlabel {$y_2$} [l] at 63 90.5
\pinlabel {$y_0$} [r] at 82.5 90.5
\pinlabel {$y_0$} [l] at 103.5 90.5
\pinlabel {$y_2$} [r] at 123 90.5
\pinlabel {$y_2$} [l] at 144 90.5
\pinlabel {$y_0$} [r] at 163.5 90.5
\pinlabel {$y_0$} [l] at 184.5 90.5
\pinlabel {$y_1$} [r] at 204.5 90.5
\pinlabel {$y_1$} [l] at 225 90.

\pinlabel {$x_0$} [t] at 12 79.5
\pinlabel {$x_0$} [t] at 53 79.5
\pinlabel {$x_1$} [t] at 93 79.5
\pinlabel {$x_1$} [t] at 133 79.5
\pinlabel {$x_2$} [t] at 174 79.5
\pinlabel {$x_2$} [t] at 214 79.5

\pinlabel {$x_0$} [b] at 50 64.5
\pinlabel {$x_1$} [b] at 122 64.5
\pinlabel {$x_2$} [b] at 194 64.5

\pinlabel* {$a_0$} at 40.5 54
\pinlabel* {$a_1$} at 113 54
\pinlabel* {$a_2$} at 184.5 54

\pinlabel {$x_1$} [r] at 21 54
\pinlabel {$x_1$} [l] at 61 54
\pinlabel {$x_2$} [r] at 93 54
\pinlabel {$x_2$} [l] at 133 54
\pinlabel {$x_0$} [r] at 165 54
\pinlabel {$x_0$} [l] at 205 54

\pinlabel {$x_0$} [t] at 32 43.5
\pinlabel {$x_1$} [t] at 104 43.5
\pinlabel {$x_2$} [t] at 177 43.5

\pinlabel {$y_0$} [b] at 50 28.5
\pinlabel {$y_1$} [b] at 122 28.5
\pinlabel {$y_2$} [b] at 194 28.5

\pinlabel {$y_1$} [r] at 21 18.5
\pinlabel {$y_1$} [l] at 61 18.5
\pinlabel {$y_2$} [r] at 93 18.5
\pinlabel {$y_2$} [l] at 133 18.5
\pinlabel {$y_0$} [r] at 165 18.5
\pinlabel {$y_0$} [l] at 205 18.5

\pinlabel* {$b_0$} at 41 18.5
\pinlabel* {$b_1$} at 113.5 18.5
\pinlabel* {$b_2$} at 185.5 18.5

\pinlabel {$y_0$} [t] at 32 7
\pinlabel {$y_1$} [t] at 104 7
\pinlabel {$y_2$} [t] at 177 7

\endlabellist
\includegraphics[width=4.6in]{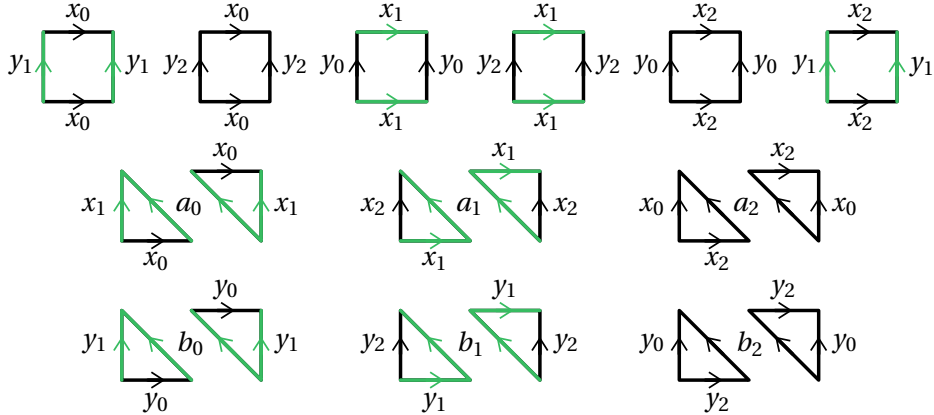}
\caption{Part of the corridor scheme ${\Ss}_{{\sigma}}$, in green, for the
  segment ${\sigma}$ from Figure \ref{fig:sigma}. The $2$--cells
  from the copy of $X$ with triple $(0,1,2)$ are shown. There is an
  analogous collection of corridor cells for every triangle that $\sigma$
  meets.}\label{fig:S-sigma} 
\end{figure}

\begin{remark}\label{corridorremark} 
Looking closely at the corridor scheme ${\Ss}_{{\sigma}}$, two additional
properties become evident. First, the ${\Ss}_{{\sigma}}$--edges appearing in a
single corridor are all labeled $x_i$ or $a_i$ for various indices $i$,
or they are all labeled $y_i$ or $b_i$. (That is, ${\Ss}_{{\sigma}}$ is the
disjoint union of two smaller corridor schemes.) 

Second, the corridor scheme ${\Ss}_{{\sigma}}$ is orientable. Referring to
Figure \ref{fig:S-sigma}, we can give positive orientations (relative to
the labeling) to the edges labeled $x_0$, $y_0$, $x_1$, $y_1$, $a_0$, and
$b_0$, and negative orientations to those labeled $a_1$ and
$b_1$. This set of choices, or its opposite, can be imposed on any copy
of $X$ in $X_T$ that contains edges of ${\Ss}_{{\sigma}}$. If two copies are
adjacent, meaning that they intersect in a subspace $Y_i$, then the
orientations on each copy can be made to agree on $Y_i$, by reversing the
choices on one side if necessary. Now recall that the copies of $X$
containing edges of ${\Ss}_{{\sigma}}$ all lie along ${\sigma}$. Starting with
the copy of $X$ at one end, one may propagate these choices consistently
over all of ${\Ss}_{{\sigma}}$. 

Orientability implies that if an ${\Ss}_{{\sigma}}$--edge label appears more
than once along a corridor, then it is oriented the same way across the
corridor in each occurrence. Furthermore, if a band type corridor
joins two ${\Ss}_{{\sigma}}$--edges in the boundary which carry the same label,
then those labels have opposite orientations relative to the boundary of
the diagram. 
\end{remark}

\subsection*{Standard generators for $V_T$}
Recall that $V_T$ contains many free subgroups $A_i = \langle a_i,
b_i\rangle$ which were the peripheral subgroups of the vertex groups
$V$. Some of these subgroups were assigned to edges of $T$ and
amalgamated together; these subgroups of $V_T$ will be called the
\emph{internal} subgroups, and their generators the \emph{internal
  generators}. Recall that every $A_i$ that is not an internal subgroup
is called a peripheral subgroup of $V_T$. 

The \emph{standard generating set} for $V_T$ will be the union of the
generators of the vertex groups (all the generators $x_i$ and $y_i$) and
the generators $a_i, b_i$ of the peripheral subgroups. The internal
generators are \emph{not} included. 

\begin{definition}\label{lengthdef} If $w$ is a word let $\abs{w}$ denote
  the length of $w$. We define some additional lengths for a word $w$ in
  the standard generators of $V_T$: 
\begin{itemize}
\item $\abs{w}_x$ is the number of occurrences of letters $x_i^{\pm 1}$
  (for any $i$) in $w$ 
\item $\abs{w}_y$ is the number of occurrences of letters $y_i^{\pm 1}$
  (for any $i$) in $w$ 
\item for each $i$, $\abs{w}_i$ is the number of occurrences of letters
  $a_i^{\pm 1}, b_i^{\pm 1}$ in $w$ 
\end{itemize}
Clearly, $\abs{w} = \abs{w}_x + \abs{w}_y + \sum_i
\abs{w}_i$. 

We use similar notation to count occurrences of $x^{\pm 1}$ 
and $y^{\pm 1}$ in words representing elements of $\langle x, y
\rangle$. 
\end{definition}

\begin{definition}\label{weightedlengthdef}
We also define \emph{weighted} word lengths $\aabs{w}$ similar to the
lengths above, where letters are counted with real-valued
weights. 

Recall that $\phi$ has transition matrix $M_{\phi}
= \begin{pmatrix} \, \abs{\phi(x)}_x & \abs{\phi(y)}_x \\ \abs{\phi(x)}_y &
  \abs{\phi(y)}_y \, \end{pmatrix}$ with Perron-Frobenius eigenvalue
$\lambda > 1$. Let $\vec{d}$ be a left eigenvector for $\lambda$ (so
that $\vec{d} M_{\phi} = \lambda \vec{d}$) with positive entries $d_1$
and $d_2$. 

To define the weighted word lengths, we assign the weight $d_1$ to the
letters $x_i$ and $a_i$, and we assign $d_2$ to each $y_i$ and
$b_i$. Thus, 
\begin{itemize}
\item $\aabs{w}_x = d_1 \abs{w}_x$ 
\item $\aabs{w}_y = d_2 \abs{w}_y$ 
\end{itemize}
The weighted length functions are needed for the sake of Lemma
\ref{minstretchmap} below. Up to scaling, this is the only choice of
weights for which the conclusion of the lemma holds. 
\end{definition}

\begin{lemma}\label{minstretchmap}
Suppose $w$ is a word in the free group $\langle x, y \rangle$ and $v$ is
the reduced word representing $\phi(w)$. Let $\aabs{ \, \cdot \, }$ denote
the weighted word length which assigns weight $d_1$ to $x^{\pm
  1}$ and weight $d_2$ to $y^{\pm 1}$, where $\vec{d} M_{\phi} = \lambda
\vec{d}$. Then $\aabs{v} \leq \lambda \aabs{w}$. 
\end{lemma}

\begin{proof}
This is a simple calculation:
\begin{align*}
\aabs{\phi(w)} \ &= \ d_1 \abs{\phi(w)}_x + d_2 \abs{\phi(w)}_y \\
&= \ d_1 \abs{\phi(x)}_x \abs{w}_x + d_1 \abs{\phi(y)}_x
  \abs{w}_y + d_2 
\abs{\phi(x)}_y \abs{w}_x + d_2 \abs{\phi(y)}_y \abs{w}_y \\
&= \ \lambda d_1 \abs{w}_x + \lambda d_2 \abs{w}_y \\
&= \ \lambda \aabs{w}.
\end{align*}
Now, $\aabs{v} \leq \aabs{\phi(w)} = \lambda\aabs{w}$. 
\end{proof}

\subsection*{The balancing property} 
A fundamental property of $V_T$ and its standard generating set, the
\emph{balancing property}, is given in the next proposition. 

\begin{proposition} 
\label{balanced}
Suppose $w$ and $z(a_{\nu_0}, b_{\nu_0})$ represent the same element of
$A_{\nu_0} \subset V_T$, where $w$ is a word in the standard generators
of $V_T$ and $z(a_{\nu_0}, b_{\nu_0})$ is reduced. Then for each $i = 1,
\dotsc, m$ there is an inequality \[\abs{z} \ \leq \ \abs{w}_{\nu_i} +
\abs{w}_{\nu_0} + \abs{w}_x + \abs{w}_y.\]
\end{proposition}

\begin{remarks}\label{balancedremarks}
(1) The proposition says that an element of a peripheral
subgroup cannot be expressed efficiently using generators from
\emph{other} peripheral subgroups. For instance, if $w$ contains only
generators from $A_{\nu_1}, \dotsc, A_{\nu_m}$, then $\abs{w}_{\nu_0} =
\abs{w}_x = \abs{w}_y = 0$ and  $\abs{w}_{\nu_i}
\geq \abs{z}$ for every $i$, whence $\abs{w} \geq m \abs{z}$. An
example of such a word $w$ is given in \eqref{standardword} below, where
$z$ is the initial subword $w(a_{\nu_0},b_{\nu_0})$ and $w$ is the
inverse of the remaining expression (see also Figure
\ref{fig:canonical}). 

(2) There is nothing special about $\nu_0$. By re-indexing the
peripheral subgroups, there is a corresponding statement that holds for
each peripheral subgroup of $V_T$. 
\end{remarks}

\begin{proof}[Proof of Proposition \ref{balanced}] 
Let ${\sigma}$ be the maximal segment in $\widehat{T}$ with endpoints
$v_{\nu_0}$ and $v_{\nu_i}$. The corridor scheme ${\Ss}_{{\sigma}}$ contains
exactly four edges whose labels are peripheral generators of $V_T$; these
generators are $a_{\nu_0}$, $b_{\nu_0}$, $a_{\nu_i}$, and
$b_{\nu_i}$. Every other standard generator occurring as the label of an
edge in ${\Ss}_{{\sigma}}$ is of the form $x_j$ or $y_j$. 

We may assume without loss of generality that $w$ is reduced. We may
further assume that the word $zw^{-1}$ is cyclically reduced, since
cancellation of letters between $z$ and $w^{-1}$ does not change the
status of the inequality. 

Let $\Delta$ be a reduced van Kampen diagram over $X_T$ with boundary
labeled by $z w^{-1}$. We may assume that $\Delta$ is topologically a
disk. Every edge on $z$ is an
${\Ss}_{{\sigma}}$--edge, and is joined by a ${\sigma}$--corridor to another
${\Ss}_{{\sigma}}$--edge on the boundary of $\Delta$. If this latter edge is
not in $z$ then it contributes $1$ to the right hand side of the
inequality, since it is labeled by a standard generator. 

We claim that no ${\sigma}$--corridor can join two edges of $z$. Then,
since ${\sigma}$--corridors never have ${\Ss}_{{\sigma}}$--edges in common,
there will be at least $\abs{z}$ ${\Ss}_{{\sigma}}$--edges along $w^{-1}$,
which establishes the result. 

If a ${\sigma}$--corridor joins two edges of $z$, then since corridors do
not cross, there is an innermost such corridor. The ${\Ss}_{{\sigma}}$--edges
that it joins must be adjacent edges of $z$, by the innermost
property. Suppose (without loss of generality) the label on one of the
edges is $a_{\nu_0}$. By Remark \ref{corridorremark} the label on the
other edge must then be $a_{\nu_0}^{-1}$, but this contradicts the
assumption that $z$ is reduced. 
\end{proof}

\begin{remark}\label{weightedremark}
Proposition \ref{balanced} remains true if \emph{weighted} word lengths
are used throughout:
\[\aabs{z} \ \leq \ \aabs{w}_{\nu_i} + \aabs{w}_{\nu_0} + \aabs{w}_x +
\aabs{w}_y\]
for each $i = 1, \dotsc, m$. Recall that the proof entailed finding
corridors joining letters of $z$ to letters of $w$. The letters occurring
at the ends of such a corridor will have the same weights, by Remark
\ref{corridorremark}. Therefore, each contribution to the left hand side
of the inequality has a matching contribution on the right hand side. 
\end{remark}

\section{Canonical diagrams} \label{sec:canonical} 

In this section we construct a large family of van Kampen diagrams over
$X_T$ called canonical diagrams. These will be used in the construction
of snowflake diagrams in the next section. We also develop properties of 
$\sigma$--corridors in order to show that canonical diagrams and
snowflake diagrams minimize area relative to their boundaries. 

\subsection*{Canonical diagrams over $X_T$}
Let $w(x,y)$ be a palindromic word in the free group. In $V_T$, for each
$i$, one has the relation $w(a_i, b_i) = w(\overline{x}_i,
\overline{y}_i) w(x_{i+1},y_{i+1})$ where ``$i+1$'' is interpreted
appropriately. Since $w$ is palindromic, this relation is identical to
the relation $w(a_i, b_i) = w(x_i, y_i)^{-1} w(x_{i+1},y_{i+1})$. It
bounds a triangular van Kampen diagram over $X_T$ of area $\abs{w}^2$;
see Figure \ref{fig:triangle}. 
\begin{figure}[ht]
\labellist
\hair 2pt
\normalsize
\pinlabel {$w(x_0, y_0)$} [t] at 38.5 2.5
\pinlabel {$w(x_{1}, y_{1})$} [r] at 2 39
\pinlabel {$w(a_0, b_0)$} [Bl] at 43 45

\pinlabel {$w(a_0, b_0)$} [t] at 136 6
\pinlabel {$w(a_1, b_1)$} [Bl] at 158 41
\pinlabel {$w(a_2, b_2)$} [Br] at 120 45

\endlabellist 
\includegraphics[width=1.5in]{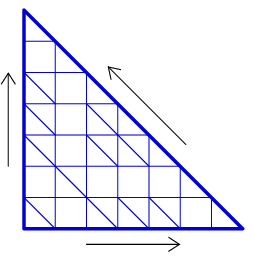} \hspace{.5in}
\includegraphics[width=1.5in]{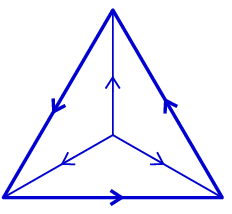}
\caption{Assembling diagrams over $X_T$}\label{fig:triangle}
\end{figure}
Assembling three 
such diagrams in cyclic fashion, one obtains, for each vertex group $V$
with index triple $(i,j,k)$, a diagram of area $3\abs{w}^2$ with boundary
word $w(a_i, b_i) w(a_j,b_j) w(a_k, b_k)$. Finally, taking diagrams of
the latter kind, one for each vertex of $T$, and assembling them
according to $T$ (just like the triangles in $D$) one
obtains a van Kampen diagram over $X_T$ of area $3\abs{T} \abs{w}^2$ with
boundary word 
\begin{equation}\label{standardword}
w(a_{\nu_0},b_{\nu_0}) w(a_{\nu_1},b_{\nu_1}) \dotsm
w(a_{\nu_m},b_{\nu_m}). 
\end{equation}
See Figure \ref{fig:canonical}. 
\begin{figure}[ht] 
\labellist
\hair 2pt
\normalsize
\pinlabel {$w(a_0, b_0)$} [t] at 18 3.5
\pinlabel {$w(a_9,b_9)$} [t] at 49 3.5
\pinlabel {$w(a_{13},b_{13})$} [tl] at 74 21
\pinlabel {$w(a_{15},b_{15})$} [Bl] at 74.5 48
\pinlabel {$w_{16}$} [Br] at 60.5 57
\pinlabel {$w_6$} [Bl] at 35 57
\pinlabel {$w(a_7,b_7)$} [Br] at 23 49
\pinlabel {$w(a_2,b_2)$} [Br] at 7 23
\endlabellist 
\includegraphics[width=1.8in]{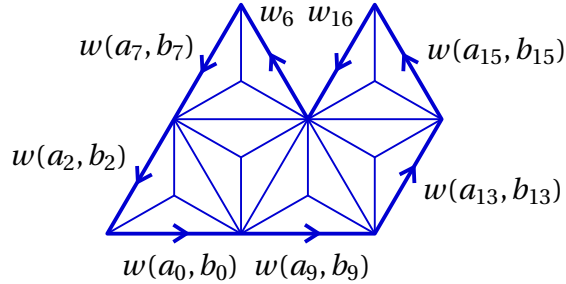}
\caption{The canonical diagram of $w$ with boundary word
  \eqref{standardword}, where $w_i = w(a_i, b_i)$.}\label{fig:canonical} 
\end{figure}
In assembling this diagram, we are using the fact that
$w(a_i, b_i) = w(a_j,b_j)^{-1}$ whenever $A_i$ and $A_j$ were
amalgamated, which also relies on the palindromic property of $w$. 

This van Kampen diagram will be called the \emph{canonical 
diagram of $w$}, and it is defined for every palindromic word. If $w$ is
reduced then the canonical diagram is also reduced. 

Our remaining objective in this section is to show that canonical
diagrams (and their ``doubled'' variants) minimize area. To this end, we
need to establish some additional properties of ${\sigma}$--corridors. 

\begin{lemma}\label{crossingsquares}
Let $\Delta$ be a van Kampen diagram over $X_T$ and suppose that $C$ is a
${\sigma}$--corridor and $C'$ is a ${\sigma}'$--corridor in $\Delta$. If\/ $C$
and $C'$ have intersection of positive area, then $C \cap C'$ contains
one of the following: 
\begin{enumerate}
\item \label{cs1} a quadrilateral relator 
\item \label{cs2} two neighboring triangular relators with a
  common ${\Ss}_{{\sigma}}\cap {\Ss}_{{\sigma}'}$--edge labeled $a_i$ or $b_i$ 
\item \label{cs3} a triangular relator with a side labeled by
  $a_i$ or $b_i$, which is a ${\Ss}_{{\sigma}} \cap {\Ss}_{{\sigma}'}$--edge in the
  boundary of\/ $\Delta$. 
\end{enumerate}
\end{lemma}

We will refer to the quadrilateral relator in \eqref{cs1} and the union of
the two neighboring triangular relators in \eqref{cs2} as \emph{crossing
  squares} for $C \cap C'$. The triangular relators in \eqref{cs3} will
be called \emph{crossing triangles}. Note that crossing squares have area
$2$. 

We shall see that canonical diagrams are completely filled by crossing
squares and triangles for various pairs of corridors, and that these
crossing regions must be present in any diagram with the same boundary. 

\begin{proof}
If case \eqref{cs1} does not occur, then $C \cap C'$ contains a
triangular relator. Note that ${\Ss}_{{\sigma}}$ has the property that a
corridor cell is triangular if and only if one of its boundary
${\Ss}_{{\sigma}}$--edges is labeled $a_i$ or $b_i$; see Figure
\ref{fig:S-sigma}. The same is true of ${\Ss}_{{\sigma}'}$. Thus, the side of
the triangular relator labeled $a_i$ or $b_i$ is an ${\Ss}_{{\sigma}}\cap
{\Ss}_{{\sigma}'}$--edge. If this edge is in the boundary of $\Delta$ then case
\eqref{cs3} occurs. Otherwise, the neighboring $2$--cell across that edge
is a corridor cell for both ${\Ss}_{{\sigma}}$ and ${\Ss}_{{\sigma}'}$ and case
\eqref{cs2} occurs. 
\end{proof}

\begin{lemma}\label{emptyint}
If ${\sigma}$ and ${\sigma}'$ are maximal segments in $\widehat{T}$ with no
edges in common, then ${\Ss}_{{\sigma}} \cap {\Ss}_{{\sigma}'}$ is empty and no
$2$--cell is a corridor cell for both ${\Ss}_{{\sigma}}$ and ${\Ss}_{{\sigma}'}$. 
\end{lemma}

\begin{proof}
Because $\widehat{T}$ has valence at most $3$, ${\sigma}$ and ${\sigma}'$
must actually be disjoint. Thus, they never pass through the same
triangle of $D$, which shows that ${\Ss}_{{\sigma}} \cap {\Ss}_{{\sigma}'}$ is
empty. It follows immediately that no triangular $2$--cell can be a
corridor cell for both ${\Ss}_{{\sigma}}$ and ${\Ss}_{{\sigma}'}$. The same is true
for quadrilateral $2$--cells because each such $2$--cell has all of its
side labels coming from a single triangle in $D$. 
\end{proof}

\begin{definition}
For each edge $e$ in $\widehat{T}$ choose maximal segments ${\sigma}_e$,
${\sigma}_e'$ in $\widehat{T}$ whose intersection is exactly $e$. If $e$ is
an interior edge, we also require that the endpoints of ${\sigma}_e$ and
${\sigma}_e'$ are linked in the boundary of the $(\abs{T}+2)$--gon; see
Figure \ref{fig:delta-e}. 

If $C$ and $C'$ are ${\sigma}_e$-- and ${\sigma}_e'$--corridors
respectively, a crossing square or crossing triangle for $C \cap C'$ will
be called an \emph{$e$--crossing square} or an \emph{$e$--crossing
  triangle} (or \emph{$e$--crossing region} in either case). Figure
\ref{fig:delta-e} shows the location of $e$--crossing regions in a
canonical diagram $\Delta$. 
\end{definition}

\begin{figure}[ht]
\labellist
\hair 2pt
\large
\pinlabel* {$e$} at 41.5 14
\pinlabel {$\textcolor[HTML]{8a2be2}{\sigma_{\! e}}$} [tl] at 49 5.5
\pinlabel {$\textcolor[HTML]{8a2be2}{\sigma_{\! e}}$} [Br] at 25 52
\pinlabel {$\textcolor{red}{\sigma'_{\! e}}$} [Br] at 6 20
\pinlabel {$\textcolor{red}{\sigma'_{\! e}}$} [Bl] at 72 51
\endlabellist 
\includegraphics[width=4.5in]{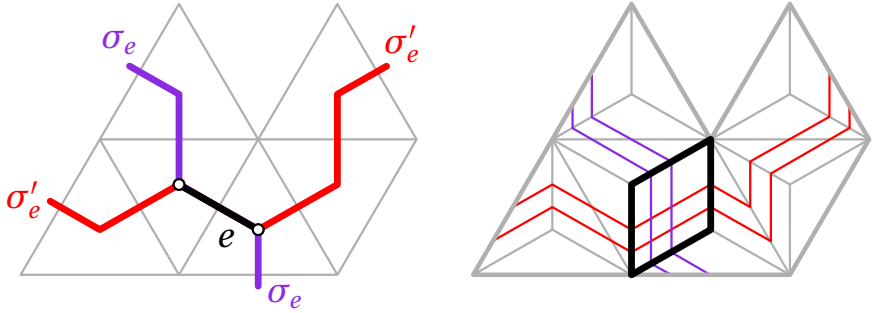}
\caption{An edge $e$ in $\widehat{T}$ and the segments ${\sigma}_e$,
  ${\sigma}_e'$. The $e$--crossing squares in $\Delta$ fill a quadrilateral
  region of area $2\abs{w}^2$ as shown. If $e$ were a peripheral edge,
  the $e$--crossing squares and triangles would fill a triangular region
  next to the boundary of $\Delta$, of area $\abs{w}^2$.}\label{fig:delta-e} 
\end{figure}

\begin{lemma}\label{e-f-regions}
If\/ $e$ and $f$ are distinct edges of $\widehat{T}$ then $e$--crossing
regions and $f$--crossing regions have no $2$--cells in common. 
\end{lemma}

\begin{proof}
It suffices to show that no $2$--cell is a corridor cell simultaneously
for all four corridor schemes ${\Ss}_{{\sigma}_e}$, ${\Ss}_{{\sigma}_e'}$,
${\Ss}_{{\sigma}_f}$, ${\Ss}_{{\sigma}_f'}$. 

If $e$ and $f$ are separated by a third edge $g$, then at least one of
${\sigma}_e$, ${\sigma}_e'$ and one of ${\sigma}_f$, ${\sigma}_f'$ does not
contain $g$. Hence, these two segments are disjoint and Lemma
\ref{emptyint} applies. 

Otherwise, $e$ and $f$ have a common vertex $v$. Consider the triangle in
$D$ centered at $v$. There are three ways that a segment ${\sigma}$ can
pass through the triangle, and the four segments must use 
all three of these. Of the eighteen $2$--cells associated with this
triangle, one can verify easily that each $2$--cell is a corridor cell
for exactly two of the three possible schemes. It follows that no
$2$--cell associated with this triangle can be a corridor cell for all
four corridor schemes. No other triangle in $D$ can meet all four
segments, so the same is true for the other $2$--cells of $X_T$. 
\end{proof}

\begin{definition}
A van Kampen diagram over $X$ is called \emph{least-area} if it has the
smallest area of all van Kampen diagrams over $X$ having the same
boundary word. 
\end{definition}

\begin{proposition}\label{leastarea}
Let $w(x,y)$ be a reduced palindromic word. The canonical diagram of $w$
is least-area. 
\end{proposition}

\begin{proof}
Let $\Delta$ be the canonical diagram of $w$ and let $\Delta'$ be an
arbitrary van Kampen diagram with the same boundary word. 
Let $\ell = \abs{w}$. 

First we claim that for any choice of ${\sigma}$, say with endpoints
$v_{i}$ and $v_{j}$, there are exactly $\ell$ band type
${\sigma}$--corridors in $\Delta'$, each joining a letter in
$w(a_{i},b_{i})$ with a letter in $w(a_{j}, b_{j})$. Certainly, these two
subwords of \eqref{standardword} contain the only occurrences of
${\Ss}_{{\sigma}}$--edges in the boundary of $\Delta'$, so the number of such
corridors can only be $\ell$. Also, no such corridor can join two
letters of the same subword $w(a_i, b_i)$ or $w(a_j,b_j)$; using Remark
\ref{corridorremark} as in the proof of Proposition \ref{balanced} one
finds that $w$ must then fail to be reduced.  

Now let us identify crossing squares and triangles in $\Delta'$. If $e$
is an internal edge of $\widehat{T}$ then every ${\sigma}_e$--corridor
crosses every ${\sigma}_e'$--corridor, by the linking requirement on
${\sigma}_e$ and ${\sigma}_e'$ (cf. Figure \ref{fig:delta-e}). Thus there are
exactly $\ell^2$ $e$-crossing squares for such $e$. 

If $e$ is a peripheral edge of $\widehat{T}$ incident to $v_i$, say, then
some pairs of ${\sigma}_e$-- and ${\sigma}_e'$--corridors cross and some do
not. There is one corridor of each type (${\sigma}_e$ or ${\sigma}_e'$)
emanating from each letter in the subword $w(a_i,b_i)$ on the boundary of
$\Delta'$, and this accounts for all ${\sigma}_e$-- and 
${\sigma}_e'$--corridors. For each letter in $w(a_i,b_i)$ the two corridors
emanating there will contain an $e$--crossing triangle; there are $\ell$ such
corridor pairs. Of the remaining corridor pairs, half of them definitely 
cross (because their endpoints on the boundary are linked), yielding
$e$--crossing squares. There are at least $\ell(\ell-1)/2$ of these. In
total we have identified $e$-crossing regions of total area $2\ell^2$
when $e$ is an internal edge, and of total area $\ell^2$ when $e$ is a
peripheral edge. By Lemma \ref{e-f-regions} we conclude that 
$\area(\Delta') \geq \area(\Delta)$. 
\end{proof}

\subsection*{Doubled canonical diagrams}
For any palindromic word $w$, take the canonical diagrams of $w$ and of
$w^{-1}$ and join them along their boundary subwords labeled
$w(a_{\nu_0},b_{\nu_0})$ and $w(a_{\nu_0},b_{\nu_0})^{-1}$ to form a new
diagram, called the \emph{doubled canonical diagram of $w$}. Its boundary
word is given by 
\begin{equation}\label{doubledword}
w(a_{\nu_1}, b_{\nu_1}) \dotsm w(a_{\nu_m}, b_{\nu_m}) w(a_{\nu_1},
b_{\nu_1})^{-1} \dotsm w(a_{\nu_m}, b_{\nu_m})^{-1}. 
\end{equation}
If $w$ is reduced, then so is its doubled canonical diagram. 

\begin{proposition}\label{doubledarea}
Let $w(x,y)$ be a reduced palindromic word. The doubled canonical diagram
of $w$ is least-area. 
\end{proposition}

\begin{proof}
Let $D\Delta$ be the doubled canonical diagram of $w$ and let $\Delta'$
be an arbitrary van Kampen diagram with the same boundary word. 

Let ${\sigma}$ be a maximal segment in $\widehat{T}$ that does not contain
$v_{\nu_0}$. The ${\Ss}_{{\sigma}}$--edges on the boundary of
$\Delta'$ comprise four subwords $w(a_i,b_i)$, $w(a_j,b_j)$,
$w(a_i,b_i)^{-1}$, $w(a_j,b_j)^{-1}$, arranged in this cyclic ordering. 
We claim that the ${\sigma}$--corridors joining letters in these subwords
must in fact join all the letters of $w(a_i,b_i)$ to those of $w(a_j,
b_j)$, and similarly with $w(a_i,b_i)^{-1}$ and $w(a_j,b_j)^{-1}$. 

First, as before, no ${\sigma}$--corridor joins two letters of the same
subword, because $w$ is reduced. Next, no corridor runs between
$w(a_i,b_i)$ and $w(a_i,b_i)^{-1}$ (or $w(a_j,b_j)$ and
$w(a_j,b_j)^{-1}$) because then there is no room for the remaining
corridors to be disjoint. Thus, the ${\sigma}$--corridors must be arranged
as in Figure \ref{fig:corridor-config}. 
\begin{figure}[ht]
\labellist
\hair 2pt
\small
\pinlabel {$w_i$} [Bl] at 103 92
\pinlabel {$w_j$} [tl] at 104 20
\pinlabel {$w_j$} [Br] at 10 94
\pinlabel {$w_i$} [tr] at 9 19

\endlabellist
\includegraphics{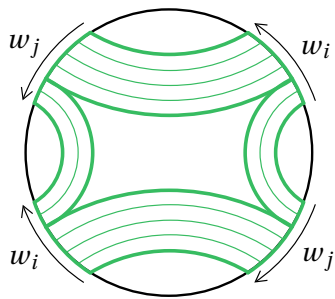}
\caption{A configuration of
  $\sigma$--corridors.}\label{fig:corridor-config} 
\end{figure}
It is evident that if any corridor joins $w(a_i,b_i)$ to
$w(a_j,b_j)^{-1}$, then there is such a corridor joining the \emph{first}
letter of $w(a_i,b_i)$ to the \emph{last} letter of $w(a_j,b_j)^{-1}$. If
the first letter is, say, $a_i$ (in the orientation of the boundary of
$\Delta'$), then the corridor joins it to $a_j^{-1}$. On the other hand,
in the \emph{canonical} diagram of $w$, there is a corridor joining $a_i$
in the boundary to the last letter of $w(a_j,b_j)$, which is $a_j$
(because $w$ is palindromic). The existence of both corridors, even in
different diagrams, contradicts the orientability of $\Ss_{\sigma}$
established in Remark \ref{corridorremark}. Therefore all corridors join
$w(a_i,b_i)$ to $w(a_j,b_j)$ or $w(a_i,b_i)^{-1}$ to $w(a_j,b_j)^{-1}$,
as claimed. 

If ${\sigma}$ is a maximal segment with endpoints $v_{\nu_0}$ and $v_i$
then the only ${\Ss}_{{\sigma}}$--edges on the boundary are the two subwords
$w(a_i,b_i)$ and $w(a_i,b_i)^{-1}$, and all ${\sigma}$--corridors run
between them.  

The rest of the proof now proceeds without difficulty just like
Proposition \ref{leastarea}. We have complete knowledge of which pairs of
edges in the boundary of $\Delta'$ are joined by corridors of various
kinds, and these pairings are in agreement with those of $D\Delta$. One
easily finds the requisite numbers of $e$--crossing regions for each $e$
and concludes that $\area(\Delta') \geq \area(D\Delta)$. 
\end{proof}

\section{Snowflake diagrams}\label{sec:lower}

Snowflake diagrams, defined below, will be used to establish the lower
bound in the proof of Theorem \ref{dehnfunction}. 

\subsection*{The $2$--complex $Y_{T,n}$} Recall that $S_{T,n}$ was defined
via the relative presentation \eqref{smn}. Starting with $X_T$, adjoin
$1$--cells and $2$--cells according to this relative presentation to
obtain the $2$--complex $Y_{T,n}$ with fundamental group $S_{T,n}$. There
will be $m$ new $1$--cells labeled $r_1, \dotsc, r_m$ and $2m$ new
$2$--cells with boundary words given by the relators of \eqref{smn}. 

\subsection*{$r_i$--corridors} For each $i = 1, \dotsc, m$ there is an
orientable corridor scheme consisting of the single edge labeled
$r_i$. It has two corridor cells which we think of as being rectangular,
with sides labeled by the words 
\begin{equation}\label{word1}
\phi^n(a_{\nu_i})^{-1}, \ r_i, \ a_{\nu_0},\ r_i^{-1}
\end{equation}
and 
\begin{equation}\label{word2}
\phi^n(b_{\nu_i})^{-1}, \ r_i, \ b_{\nu_0},\ r_i^{-1}.
\end{equation}
The sides labeled by $a_{\nu_0}$ or $b_{\nu_0}$ will be called the
\emph{short sides} and the sides labeled by $\phi^n(a_{\nu_i})^{-1}$ or
$\phi^n(b_{\nu_i})^{-1}$ the \emph{long sides} of the corridor cells. The
corridors for this scheme are called \emph{$r_i$--corridors}. 

Note that in any $r_i$--corridor in a \emph{reduced} van Kampen diagram,
the short sides of the corridor cells join up to form a single arc in the
boundary of the corridor, labeled by a reduced word $w$ in the generators
$a_{\nu_0}$, $b_{\nu_0}$. If this word happens to be monotone, then the
long sides of the corridor cells also assemble to form a monotone (and
reduced) word $\phi^n(w)(a_{\nu_i}, b_{\nu_i})$.

\subsection*{Snowflake diagrams} Let $w(x,y)$ be a \emph{monotone}
palindromic word. We will define van Kampen diagrams over 
$Y_{T,n}$ based on $w$ and an integer $d$ (the \emph{depth}) denoted
$\Delta(w,d)$. To begin, we define $\Delta(w,0)$ to be the doubled
canonical diagram of $w$. 

Next, to define $\Delta(w,d)$ for $d > 0$, start with the diagram
$\Delta(\phi^n(w),d-1)$ (noting that $\phi^n(w)$ is also monotone 
and palindromic, by our assumptions on $\phi$). Its boundary word 
will have subwords of the form $\phi^n(w)(a_{\nu_i},b_{\nu_i})^{\pm
  1}$ for each $i = 1, \dotsc, m$. Alongside each subword
$\phi^n(w)(a_{\nu_i},b_{\nu_i})^{\epsilon}$ adjoin a rectangular strip
made of $\abs{w}$ $2$--cells whose four sides are labeled by the words 
\[\phi^n(w)(a_{\nu_i},b_{\nu_i})^{-\epsilon}, \ r_i,
\ w(a_{\nu_0},b_{\nu_0})^{\epsilon}, \ r_i^{-1}.\] Then, adjoin a copy of
the canonical diagram of $w^{-\epsilon}$, which contains a side labeled 
$w(a_{\nu_0}, b_{\nu_0})^{-\epsilon}$. 

Doing this for each subword as described, one obtains $\Delta(w,d)$. See
Figure \ref{fig:snowflake}. Note that the boundary of $\Delta(w,d)$
contains many copies of the subwords $w(a_{\nu_i},b_{\nu_i})^{\pm 1}$ for
each $i = 1, \dotsc, m$. In fact, it is easy to verify by induction on
$d$ that the boundary word is made entirely of copies of these words,
together with occurrences of the letters $r_i^{\pm 1}$. 
\begin{figure}[ht]
\labellist
\hair 2pt
\small
\pinlabel {$\textcolor[HTML]{0000d0}{\vdots}$} [B] at 146.5 113
\pinlabel {$\textcolor[HTML]{0000d0}{\vdots}$} [t] at 216.5 50.5
\pinlabel* {\rotatebox[origin=c]{144}{$\textcolor[HTML]{0000d0}{\vdots}$}}
at 98.5 138
\pinlabel* {\rotatebox[origin=c]{36}{$\textcolor[HTML]{0000d0}{\vdots}$}}
at 26 142.5
\pinlabel* {\rotatebox[origin=c]{36}{$\textcolor[HTML]{0000d0}{\vdots}$}}
at 178 22
\pinlabel* {\rotatebox[origin=c]{108}{$\textcolor[HTML]{0000d0}{\vdots}$}}
at 1 68
\pinlabel* {\rotatebox[origin=c]{72}{$\textcolor[HTML]{0000d0}{\vdots}$}}
at 2 10

\pinlabel {$w_0$} [tl] at 251 75
\pinlabel {$w_1$} [tr] at 261 49
\pinlabel {$w_2$} [tl] at 284 49.5
\pinlabel {$w_3$} [Bl] at 293.5 75
\pinlabel {$w_4$} [B] at 271 95.5

\pinlabel {$w_0$} [Bl] at 237 120
\pinlabel {$w_1$} [tl] at 262.5 114.5
\pinlabel {$w_2$} [Bl] at 272.5 140
\pinlabel {$w_3$} [B] at 250 161
\pinlabel {$w_4$} [Br] at 230 148

\pinlabel {$w_0$} [Br] at 196 121
\pinlabel {$w_1$} [Bl] at 200.5 148
\pinlabel {$w_2$} [B] at 182 161
\pinlabel {$w_3$} [Br] at 158.5 141
\pinlabel {$w_4$} [tr] at 165 118

\pinlabel {$\phi^n w_2$} [Br] at 242 77
\pinlabel {$\phi^n w_0$} [Bl] at 187.5 82
\pinlabel {$\phi^{2n} w_3$} [tr] at 178.5 81
\pinlabel {$\phi^{2n} w_0$} [tl] at 112 74
\pinlabel {$\phi^{3n} w_1$} [Br] at 102 76.6
\pinlabel {$\phi^{3n} w_0$} [B] at 63 52
\pinlabel {$\phi^{4n} w_3$} [t] at 63 37

\endlabellist
\includegraphics[width=5in]{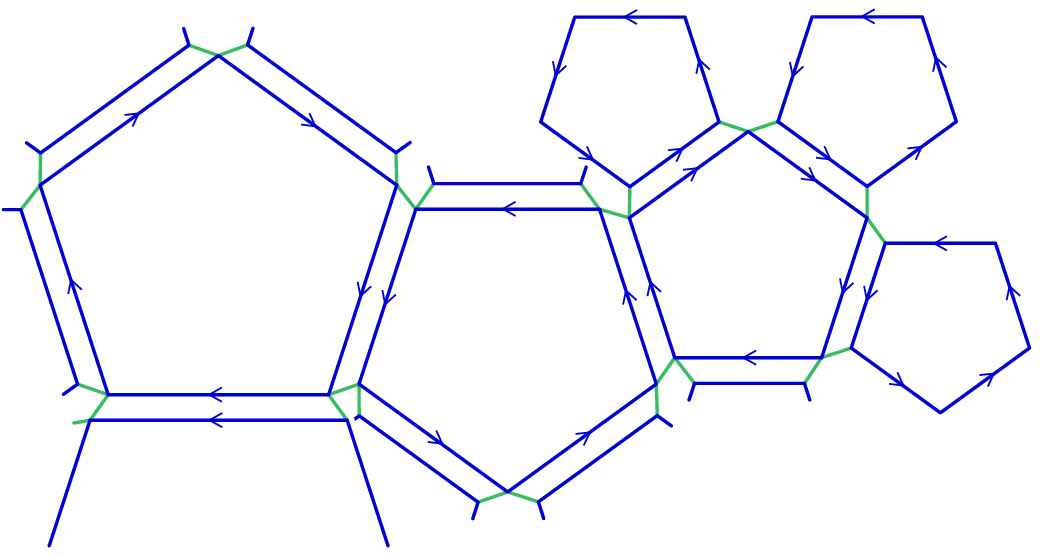}
\caption{Part of the snowflake diagram $\Delta(w,d)$, with   $m=4$ (and
  $d \geq 4$). The word $\phi^{in}(w)(a_{\nu_j},b_{\nu_j})$ is
  abbreviated as $\phi^{in}w_j$. The pentagonal regions are canonical 
  diagrams and the strips between them are $r_j$--corridors for various
  $j$.}\label{fig:snowflake}  
\end{figure}

Note that $\Delta(w,d)$ contains a sub-diagram $\Delta(\phi^{in}(w),d-i)$
for each $i$ between $0$ and $d$. In particular, it contains a copy of
$\Delta(\phi^{dn}(w),0)$, which is a doubled canonical diagram of area
$6\abs{T}\abs{\phi^{dn}(w)}^2$. 

\begin{proposition}\label{sfleastarea}
Let $w(x,y)$ be a monotone palindromic word. For each $d \geq 0$ the
diagram $\Delta(w,d)$ is least-area. 
\end{proposition}

\begin{proof}
First note that $\Delta(w,d)$ is reduced, since it is made of reduced
sub-diagrams, separated by reduced $r_i$--corridors, which have no
$2$--cells in common with the sub-diagrams. (The sub-diagrams being
reduced depends on the monotonicity of $w$, which implies that the words
$\phi^{in}(w)$ are reduced.) 

Now suppose that $d=0$ and let $\Delta'$ be any reduced diagram over
$Y_{T,n}$ with the same boundary as $\Delta(w,0)$. We claim that there
are no $r_i$--corridors for any $i$. If there were, they would be of
annulus type, and the short side of the corridor would be labeled by a
cyclically reduced word $v(a_{\nu_0},b_{\nu_0})$ representing the trivial
element. Since $A_{\nu_0}$ is free on $a_{\nu_0}, b_{\nu_0}$, no such
corridors can exist. Therefore $\Delta'$ is actually a diagram over
$X_T$, and Proposition \ref{doubledarea} says that $\area(\Delta') \geq
\area(\Delta(w,0))$. 

Proceeding by induction on $d$, suppose that $d \geq 1$ and $\Delta'$ is
a reduced diagram over $Y_{T,n}$ of smallest area, with the same boundary
as $\Delta(w,d)$. As before, there
can be no $r_i$--corridors of annulus type. There will be band type
$r_i$--corridors joining occurrences of $r_i^{\pm 1}$ on the boundary. Note
that $r_i$-- and $r_j$--corridors cannot cross for any $i,j$ (no
$2$--cell is a corridor cell for both corridor schemes). Hence the
$r_i$--edges on the boundary must be paired by corridors in the same way
as in $\Delta(w,d)$. 

Consider an outermost $r_i$--corridor. Its complement in $\Delta'$ is
two sub-diagrams, one of which is a diagram over $X_T$ with boundary word 
\[ \widehat{w}(a_{\nu_0}, b_{\nu_0})^{\epsilon}
w(a_{\nu_1},b_{\nu_1})^{\epsilon} \dotsm
w(a_{\nu_m},b_{\nu_m})^{\epsilon}\] 
for some word $\widehat{w}$ and some $\epsilon = \pm 1$. Here,
$\widehat{w}(a_{\nu_0}, b_{\nu_0})$ is the word along the short side of
the corridor. Recall that in $V_T$, the word
$w(a_{\nu_1},b_{\nu_1})^{\epsilon} \dotsm
w(a_{\nu_m},b_{\nu_m})^{\epsilon}$ represents the element
$w(a_{\nu_0},b_{\nu_0})^{-\epsilon}$, which moreover is in the free
subgroup $A_{\nu_0}$. Since $\widehat{w}(a_{\nu_0},b_{\nu_0})$ is
reduced, it must equal $w(a_{\nu_0},b_{\nu_0})$. It follows that the
corridor, considered as a sub-diagram, is identical to the corresponding
corridor in $\Delta(w,d)$. Also, the part of $\Delta'$ on the short side
of the corridor is a diagram over $X_T$ with the same boundary as the
canonical diagram of $w^{\epsilon}$. By Proposition \ref{leastarea} its
area agrees with that of the canonical diagram. 

Taking all the outermost $r$--corridors and the
sub-diagrams that they separate from the central region in $\Delta'$, we
have found that these have total area equal to the corresponding regions
in $\Delta(w,d)$. Moreover, if we delete these regions, the resulting
boundary word is the boundary word of the corresponding sub-diagram
$\Delta(\phi^n(w),d-1)$ of $\Delta(w,d)$. By induction, the central
portion of $\Delta'$ has area equal to that of $\Delta(\phi^n(w),d-1)$
and we are done. 
\end{proof}

\section{Folded corridors and subgroup
  distortion} \label{sec:distortion} 

The main result of this section is Proposition \ref{distortion} (and its
variant Corollary \ref{distortioncor}) which
bounds the distortion of the edge group $A_{\nu_0}$ in $S_{T,n}$. After
discussing some preliminaries, we proceed to study \emph{folded
corridors}, culminating in Lemma \ref{segmentsnearlyabove}. This lemma
plays an important role in the proof of Proposition \ref{distortion},
which occupies the rest of the section. 

\medskip

Define the \emph{standard generating set} for $S_{T,n}$ to be
the standard generating set for $V_T$ together with the generators $r_1,
\dotsc, r_m$. Recall that the former generators include all generators
$x_i$ and $y_i$, and the peripheral generators $a_{\nu_i}$, $b_{\nu_i}$
($i = 0, \dotsc, m$). 

For $g \in A_{\nu_i}$ let $\abs{g}_{A_{\nu_i}}$ denote the length of the
\emph{reduced} word in the basis $a_{\nu_i}$, $b_{\nu_i}$ representing
$g$. Similarly, let $\aabs{g}_{A_{\nu_i}}$ be the \emph{weighted} word
length of the reduced representative (cf. Definition
\ref{weightedlengthdef}). Recall that the letters $a_{i}^{\pm 1}$ and
$x_i^{\pm 1}$ have weight $1 + \sqrt{2}$ and the letters $b_{i}^{\pm
  1}$ and $y_i^{\pm 1}$ have weight $1$. Let the letters $r_i^{\pm 1}$
also be given weight $1$. 

Now we assign lengths to the edges of $Y_{T,n}$, and correspondingly to
the edges in any van Kampen diagram over $Y_{T,n}$, as follows. Edges
labeled by $x_i$ or $a_i$ are given length $1+\sqrt{2}$, and all other
edges (those labeled $y_i$, $b_i$, or $r_i$) are given length
$1$. In this section, lengths of paths in a van Kampen diagram will
always be meant with respect to these edge lengths. 

With this convention, the length of a path in the $1$--skeleton will
agree with the weighted length of the word labeling it.

\subsection*{Folded corridors} 
Recall that each $r_i$--corridor cell has a short side and
a long side. In any $r_i$--corridor, the embedded open annulus
or open band inside it separates all of the short sides of the corridor
cells from the long sides. 

The boundary of an $r_i$--corridor is a $1$--complex containing zero or
two $r_i$--edges. The \emph{partial boundary} is defined to be the
boundary with the interiors of the $r_i$--edges removed. 

If $C$ is an $r_i$--corridor in a reduced diagram, 
then the short sides of its cells join to form a component of the partial 
boundary which is labeled by a reduced word in the generators
$a_{\nu_0}$, $b_{\nu_0}$. If $C$ were of annulus type, then we would have a
cyclically reduced word in the free group $A_{\nu_0}$ representing the
trivial element. Hence, $C$ must be of band type. 

Following \cite{BrGr}, a band type $r_i$--corridor is called 
\emph{folded} if it is reduced and every component of its partial
boundary is labeled by a reduced word in the generators of $S_{T,n}$. We 
have noted already that the short sides of corridor cells form a single
such component, which we now call the \emph{bottom} of the corridor. Any
other component is labeled by a reduced word in the generators
$a_{\nu_i}$, $b_{\nu_i}$. Again, such a component cannot be a loop (since
$A_{\nu_i}$ is free) and hence there is only one other component, which
we call the \emph{top} of the corridor. If $w(a_{\nu_0},b_{\nu_0})$ is
the reduced word along the bottom, then the top is labeled by the reduced
word in $a_{\nu_i}$, $b_{\nu_i}$ representing
$\phi^n(w)(a_{\nu_i},b_{\nu_i})$. 

\begin{remark}\label{corridorbuilding}
Given any reduced word $w(a_{\nu_0},b_{\nu_0})$, one can build a folded
corridor with bottom labeled by $w$. Start by joining corridor cells end
to end along $r_i$--edges to form a corridor with bottom side labeled by
$w$. Then, the long sides of the corridor cells form an arc labeled by a
possibly \emph{unreduced} word representing $\phi^n(w)(a_{\nu_i},
b_{\nu_i})$. By successively folding together adjacent pairs of edges
along the top (with matching labels), one eventually obtains a folded
corridor. Each folding operation corresponds to a free reduction in the
word labeling the top side of the corridor. The final word along the top
of the folded corridor is uniquely determined (being the reduced form of
$\phi^n(w)(a_{\nu_i},b_{\nu_i})$) but the internal structure will depend
on the particular sequence of folds chosen. 
\end{remark}

Let $C$ be a folded $r_i$--corridor. Define $S \subset C$ to be the
smallest subcomplex containing all the $r_i$--edges, and all the open
$1$--cells which lie in the interior of $C$ (informally, the \emph{seams}
in $C$). Note that $S$ contains exactly those edges of $C$ that are not
in the top or bottom. 

\begin{lemma}
Let $S_0$ be a connected component of\/ $S$. Then
\begin{enumerate}
\item \label{a1} $S_0$ is a tree; 
\item \label{a2} $S_0$ contains exactly one vertex in the top of\/ $C$; 
\item \label{a3} every valence-one vertex in $S_0$ lies in the top or
  bottom of\/ $C$. 
\end{enumerate}
\end{lemma}

Conclusions \eqref{a2} and \eqref{a3} imply that $S_0$ contains at least
one vertex in the bottom of $C$. 

\begin{proof}
First note that every $2$--cell of $C$ meets the bottom in
exactly one edge. Conclusion \eqref{a1} follows immediately since a loop
in $S_0$ would separate a $2$--cell from the bottom. For the same reason,
$S_0$ cannot contain two or more vertices of the top. 

A second observation is that the boundary of every $2$--cell is labeled by
a cyclically reduced word (namely, \eqref{word1} or \eqref{word2}). Hence
no two adjacent edges of the same cell 
can be folded together. Therefore $S_0$ cannot have a valence-one vertex
in the interior of $C$, whence \eqref{a3}. 

It remains to show that $S_0$ contains a vertex of the top. If not, then
it is separated from the top by a $2$--cell which then must meet the bottom
in a disconnected set, contradicting the initial observation above. 
\end{proof}

Let $p$ be a vertex in the top of $C$ and $q$ a vertex in the bottom. We
say that $p$ \emph{is above} $q$ if both vertices are in the same
connected component of $S$. We have the following 
``bounded cancellation'' lemma, which is a restatement of Lemma 1.2.4 of
\cite{BrGr}: 

\begin{lemma}\label{boundedcancellation}
There is a constant $K_0 = K_0(\phi^n)$ such that if $p$ is a vertex in the
top of a folded corridor and $q_1$, $q_2$ are
vertices that are both below $p$, then the sub-segment 
$[q_1, q_2]$ of the bottom 
has at most $K_0$ edges. \qed
\end{lemma}

For vertices $p$ in the top and $q$ in the bottom,  we say that $p$ 
\emph{is nearly above} $q$ if there is a vertex $p'$ above $q$ such that
$p$ and $p'$ are in the boundary of a common $2$--cell. It is clear that
every vertex in the top is nearly above some vertex on the bottom. 

\begin{lemma}\label{nearlyabove}
There is a constant $K_1 = K_1(\phi^n)$ such  that if\/ $C$ is a folded
corridor and $p$ is nearly above $q$ in $C$, then there is a path
in the $1$--skeleton of\/ $C$ from $p$ to $q$, containing no bottom
edges, of length at most $K_1$. 
\end{lemma}

\begin{proof}
Let $L$ be the maximum of the boundary lengths of the two $r_i$--corridor
cells. Let $p'$ be a vertex above $q$ such that $p$ and $p'$ are in the
boundary of a common $2$--cell $e^2$, and let $S_0$ be the component of
$S$ containing $p'$ and $q$. Let $q'$ be the unique bottom vertex of
$S_0$ which is in the boundary of $e^2$. There are paths $[p,p']$ in the
top and $[p',q']$ in $S_0$. These have total length at most $L$, since
they are in the boundary of $e^2$. 

Next, any two adjacent bottom vertices in $S_0$ are in the boundary of a
$2$--cell, and hence are joined by a path in $S_0$ of length
at most $L$. It follows that $q'$ and $q$ are joined by a path $[q', q]$
in $S_0$ of length at most $K_0 L$, with $K_0$ given by Lemma 
\ref{boundedcancellation}. Now the path $[p, p'] \cdot [p', q'] \cdot
[q',q]$ has length at most $K_1 = (K_0 + 1)L$. 
\end{proof}

\begin{lemma}\label{segmentsnearlyabove}
There is a constant $K_2 = K_2(\phi^n)$ such that if\/ $C$ is a folded
corridor and $[p_1,p_2]$, $[q_1,q_2]$ are sub-segments of the top and bottom,
respectively, with $p_j$ nearly above $q_j$ for $j = 1,2$, and $u(a_{\nu_i},
b_{\nu_i})$ is the word labeling $[p_1, p_2]$, then 
\[ \aabs{\phi^{-n}(u)}_{A_{\nu_i}} - \ K_2 \ \leq \ \abs{[q_1,q_2]} \ \leq \
\aabs{\phi^{-n}(u)}_{A_{\nu_i}} + \ K_2 .\] 
\end{lemma}

The essential point is that the length of $[q_1, q_2]$ is
determined, up to an \emph{additive} error, by the word $u$ labeling
$[p_1, p_2]$. (It is certainly not determined by the length of $[p_1,
p_2]$ alone.) 

\begin{proof}
Let $w(a_{\nu_0}, b_{\nu_0})$ be the reduced word labeling $[q_1, q_2]$. 
Let $\widehat{S}$ be the union of $S$ and the top of $C$. For $j = 1, 2$
let $[q_j, p_j]$ be a shortest path in $\widehat{S}$ from $q_j$ to
$p_j$. Its first edge is an $r_i$--edge labeled $r_i^{-1}$, and so the
label on $[q_j, p_j]$ has the form $r_i^{-1} \cdot
v_j(a_{\nu_i},b_{\nu_i})$ for some reduced word $v_j$. This word has
weighted length less than $K_1$ by Lemma \ref{nearlyabove}. 

The four segments form a relation in $S_{T,n}$, namely 
\begin{align}
w(a_{\nu_0},b_{\nu_0}) \ &= \ r_i^{-1} v_1(a_{\nu_i},b_{\nu_i})
u(a_{\nu_i}, b_{\nu_i}) v_2(a_{\nu_i}, b_{\nu_i})^{-1} r_i \notag \\ 
&= \ \phi^{-n}(v_1)(a_{\nu_0},b_{\nu_0}) \, 
\phi^{-n}(u)(a_{\nu_0},b_{\nu_0}) \,
\phi^{-n}(v_2)(a_{\nu_0},b_{\nu_0}).  \label{bound1} 
\end{align}
This is an equality of elements of the free subgroup $A_{\nu_0}$. 
Now define
\[ K_2 \ = \ 2 \max \left\{ \aabs{\phi^{-n}(v)}_{A_{\nu_0}} \mid v \text{
    is a word in } a_{\nu_0}, b_{\nu_0} \text{ of weighted 
    length } < K_1 \right\}.\] 
Observe that the left hand side of \eqref{bound1} has reduced weighted length
$\abs{[q_1,q_2]}$ and the right hand side has reduced weighted length
within $K_2$ of $\aabs{\phi^{-n}(u)(a_{\nu_0}, b_{\nu_0})}_{A_{\nu_0}}$,
which is equal to $\aabs{\phi^{-n}(u)}_{A_{\nu_i}}$. 
\end{proof}

We turn now to the main result of this section, the bound on edge group
distortion. Recall from the Introduction that the proof is an inductive
proof based on Britton's Lemma. It falls into two cases, requiring very
different methods. 

In the first case, the proof is based on a method from \cite{BrBr}. It is
the balancing property of $V_T$ that allows us to carry out this
argument. The crucial moment occurs in \eqref{min-j} and the choice of
index $j'$, and the subsequent reasoning. 

The second case makes use of folded corridors and Lemma
\ref{segmentsnearlyabove}. The overvall induction argument is based on
the nested structure 
of $r_j$--corridors in a van Kampen diagram. If these corridors
are always oriented in the correct direction, then the argument based
on the balancing property would suffice. If there exists a
backwards-facing $r_j$--corridor, then the inductive process will
inevitably land in Case II. 

The backwards-facing corridor may then introduce geometric effects that
adversely affect the inductive calculation. When this occurs, we prove 
that there will be correctly oriented corridors just behind the first
one, and perfectly matching segments along these corridors, along which
any metric distortion introduced by the first corridor is exactly
undone. This occurs in \eqref{matchingsegments}, using Lemma
\ref{segmentsnearlyabove}. This argument also depends crucially on
$\sigma$--corridors. 

\begin{proposition}[Edge group distortion]\label{distortion} 
Given $T$ and $n$ there is a constant $K_3$ such that if $w$ is a word in
the standard generators of $S_{T,n}$ representing an element $g \in
A_{\nu_0}$ then  
\[\aabs{g}_{A_{\nu_0}} \ \leq \ K_3 \aabs{w}^{\alpha}\]
where $\alpha = n \log_m(\lambda)$. 
\end{proposition}

\begin{proof}
Let $K_3 \ = \ \max\{1, 3K_2/2, \Lambda^{\! n}\}$
where $K_2$ is given by Lemma \ref{segmentsnearlyabove} and $\Lambda$ is
the maximum stretch factor for $\phi^{-1}$ with respect to the weighted
word length $\aabs{ \, \cdot \, }$. 

The universal cover $\widetilde{Y}_{T,n}$ is the total space of a tree of
spaces, with vertex spaces equal to copies of the universal cover
$\widetilde{X}_{T}$. Every $1$--cell of $\widetilde{Y}_{T,n}$ is either
contained in a vertex space, or is labeled $r_j^{\pm 1}$ for some $j$ and
has endpoints in two neighboring vertex spaces. 

We argue by induction on the number of occurrences of $r_j^{\pm 1}$
in $w$ (for all $j$). We may assume that $w$ is a shortest word in the
generators of $S_{T,n}$ representing $g$. This word
describes a labeled geodesic in the $1$--skeleton of 
$\widetilde{Y}_{T,n}$ with endpoints in the same vertex space. Using the
tree structure, one finds a decomposition of $w$ as $w_1 \dotsm w_k$ where
each $w_i$ satisfies one of the following:
\begin{enumerate}
\item\label{w1} $w_i = r_j u_i r_j^{-1}$ for some $j$ and $w_i$
  represents an element of $A_{\nu_j}$. Let $v_i(a_{\nu_j},b_{\nu_j})$ be
  the reduced word representing $w_i$ in this case. 
\item\label{w2} $w_i = r_j^{-1} u_i r_j$ for some $j$ and $w_i$
  represents an element of $A_{\nu_0}$. Let $v_i(a_{\nu_0},b_{\nu_0})$ be
  the reduced word representing $w_i$ in this case. 
\item\label{w3} $w_i$ is a word in the generators $a_{\nu_j}$,
  $b_{\nu_j}$ for some $j$. 
\item\label{w4} $w_i$ is a word in the generators $\{x_j, y_j\}$
  (allowing all $j$). 
\end{enumerate}
Let $v_i = w_i$ in cases \eqref{w3} and \eqref{w4}, and define $v =
v_1 \dotsm v_k$. 
The proof now splits into cases, based on the structure of the
decomposition $w = w_1 \dotsm w_k$. 

\begin{itemize}
\item {Case IA:} \ all subwords $w_i$ are of types \eqref{w3} or
  \eqref{w4}. This is simply the base case of the induction. 

\item {Case IB:} \ either there is a subword $w_i$ of type \eqref{w1}, or
  there are two or more subwords of type \eqref{w2}. 

\item {Case II:} \ exactly one subword $w_i$ is of type \eqref{w2}
  and all others are of types \eqref{w3} or \eqref{w4}. 
\end{itemize}

\noindent
\textit{Proof in Cases IA and IB.} \ 
Define the sets 
\begin{align*}
I_j \ &= \ \{\, i \mid v_i \text{ is a word in the generators } a_{\nu_j},
b_{\nu_j}\,\},\\
I_{xy} \ &= \ \{\, i \mid v_i \text{ is a word in the generators } \{x_j,
y_j\} \,\}.
\end{align*}
Note that $\aabs{v}_{\nu_j} = \sum_{i \in I_j} \aabs{v_i}$ 
and $\aabs{v}_x + \aabs{v}_y = \sum_{i \in I_{xy}} \aabs{v_i}$. 
We begin by establishing two claims. 

\medskip
\textit{Claim 1: if\/ $i \in I_j$ and $j \not= 0$ then $\aabs{v_i} \leq
  K_3(m\aabs{w_i})^{\alpha}$.} If $w_i$ satisfies \eqref{w3} then the claim
is trivial: $\aabs{v_i} = \aabs{w_i} \leq K_3(m\aabs{w_i})^{\alpha}$
since $K_3, m, \alpha \geq 1$. Otherwise, $w_i$ satisfies \eqref{w1} and
$w_i = r_j u_i r_j^{-1}$ where $u_i$ represents an element of $A_{\nu_0}$. 
Let $z_i$ be the reduced word in $a_{\nu_0}$, $b_{\nu_0}$ equal to $u_i$
in $A_{\nu_0}$, and note that $\phi^n(z_i)(a_{\nu_j}, b_{\nu_j}) = v_i$
in $A_{\nu_j}$. Since $v_i$ is reduced we have 
\[\aabs{v_i} \ \leq \ \aabs{\phi^n(z_i)(a_{\nu_j}, b_{\nu_j})} \ = \
\aabs{\phi^n(z_i)} \ \leq \
\lambda^n \aabs{z_i} \ = \ \lambda^n \aabs{u_i}_{A_{\nu_0}} \ = \
m^{\alpha} \aabs{u_i}_{A_{\nu_0}}.\]
Here, the second inequality follows from Lemma \ref{minstretchmap}. The
final quantity is at most 
\[K_3 m^{\alpha} (\aabs{w_i}-2)^{\alpha} \leq K_3
(m \aabs{w_i})^{\alpha}\]
by the induction hypothesis. 

\medskip
\textit{Claim 2: if\/ $i \in I_0 \cup I_{xy}$ then $\aabs{v_i} \leq
  K_3\aabs{w_i}^{\alpha}$.} As in Claim 1, if $w_i$ satisfies \eqref{w3} or
\eqref{w4} then Claim 2 is trivially true. The remaining case is when
$w_i$ satisfies \eqref{w2}. Now the claim is an instance of the induction
hypothesis: since $v_i$ is reduced we have $\aabs{v_i} =
\aabs{w_i}_{A_{\nu_0}} \leq K_3 \aabs{w_i}^{\alpha}$. Note that the
induction hypothesis applies precisely because we are not in Case II,
and $w_i$ contains fewer occurrences of the letters $r_j^{\pm 1}$ than
$w$. 

\medskip
Among the indices $1, \dotsc, m$, choose $j'$ to minimize the sum
$\sum_{i \in I_{j'}} \aabs{w_i}$. Thus we have 
\begin{equation}\label{min-j}
m \sum_{i \in I_{j'}} \aabs{w_i} \ \leq \ \sum_{i \in I_{1} \cup
  \dotsb \cup I_{m}} \aabs{w_i}. 
\end{equation}
Observe that $v = v_1 \dotsm v_k$ is a word in the standard generators of
$V_T$ representing the element $g$. Applying Proposition \ref{balanced}
(and Remark \ref{weightedremark}) to this word yields the inequality
\begin{align*}
\aabs{g}_{A_{\nu_0}}  &\leq \ \aabs{v}_{\nu_0} + \aabs{v}_x + \aabs{v}_y +
\aabs{v}_{\nu_{j'}} \\
&= \ \sum_{i \in I_0} \aabs{v_i} \ +  \sum_{i \in I_{xy}} \aabs{v_i} \ + 
\sum_{i \in I_{j'}} \aabs{v_i}.
\end{align*}
Then we have 
\begin{align*}
\aabs{g}_{A_{\nu_0}} \ &\leq \sum_{i \in I_0 \cup I_{xy}} K_3
\aabs{w_i}^{\alpha} \ + \sum_{i \in I_{j'}} K_3 \left( m
  \aabs{w_i}\right)^{\alpha} \\
& \leq \sum_{i \in I_0 \cup I_{xy}} K_3 \aabs{w_i}^{\alpha} \ + \ K_3
\left(\sum_{i \in I_{j'}} m\aabs{w_i}\right)^{\alpha} 
\end{align*}
by Claims 1 and 2, and 
\begin{align*}
\aabs{g}_{A_{\nu_0}} \ &\leq \sum_{i \in I_0 \cup I_{xy}} K_3
\aabs{w_i}^{\alpha} \ + \ K_3 \left( \sum_{i \in (I_1 \cup \dotsb \cup
    I_m)} \aabs{w_i} \right)^{\alpha} \\
&\leq \ K_3 \left( \sum_i \aabs{w_i}\right)^{\alpha} \\
&= \ K_3 \aabs{w}^{\alpha}
\end{align*}
By \eqref{min-j}. 

\medskip
\noindent
\textit{Proof in Case II.} \ 
In this case we write $w$ as $w_L \hat{w} w_R$ where $\hat{w}$ is
the subword of type \eqref{w2} and $w_L$, $w_R$ are products of
subwords of types \eqref{w3} and \eqref{w4}. Thus $\hat{w} =
r^{-1}_{{j'}} u r_{{j'}}$ for some index ${j'} \not=0$, where
$u$ represents an element of $A_{\nu_{{j'}}}$. Let $\hat{v}$ be
the reduced word in $a_{\nu_0}, b_{\nu_0}$ representing $\hat{w}$,
and let $\hat{u}$ be the reduced word in $a_{\nu_{{j'}}},
b_{\nu_{{j'}}}$ representing $u$. 

We proceed by decomposing the word $u$ using the tree of spaces
structure of $\widetilde{Y}_{T,n}$. First, choose an index ${j''}
\not= {j'}$ from the set $\{1, \dotsc, m\}$. Then write $u =
\hat{w}_1 \dotsm \hat{w}_{k'}$ where each $\hat{w}_i$
satisfies one of the following: 
\begin{enumerate}
\setcounter{enumi}{4}
\item\label{hw1} $\hat{w}_i = r_{{j'}} \hat{u}_i r_{{j'}}^{-1}$ or
  $\hat{w}_i = r_{{j''}} \hat{u}_i r_{{j''}}^{-1}$. 
  Let $\hat{v}_i$ be the reduced word in $a_{\nu_{j'}}, b_{\nu_{j'}}$ or
  $a_{\nu_{j''}}, b_{\nu_{j''}}$ representing $\hat{w}_i$. 
\item\label{hw2} $\hat{w}_i = r_j \hat{u}_i r_j^{-1}$ for some $j \not=
  {j'}, {j''}$. Let $\hat{v}_i(a_{\nu_{j}}, b_{\nu_{j}})$ be the reduced
  word representing $\hat{w}_i$. 
\item\label{hw3} $\hat{w}_i = r_j^{-1} \hat{u}_i r_j$ for some $j$. Let
  $\hat{v}_i(a_{\nu_{0}}, b_{\nu_{0}})$ be the reduced word representing
  $\hat{w}_i$. 
\item\label{hw4} $\hat{w}_i$ is a word in the standard generators of
  $V_T$. 
  Let $\hat{v}_i = \hat{w}_i$ in this case. 
\end{enumerate}
There is an equality $\hat{u}=\hat{v}_1 \dotsm \hat{v}_{k'}$ in
$V_T$. Let $\Delta$ be a van Kampen diagram over $X_T$ with boundary word
$\hat{v}_1 \dotsm \hat{v}_{k'} \hat{u}^{-1}$. The arcs in the boundary of
$\Delta$ labeled by the words $\hat{v}_i$ or $\hat{u}^{-1}$ are called the
\emph{sides} of $\Delta$. The side labeled by $\hat{v}_i$ is declared to
be of type \eqref{hw1}, \eqref{hw2}, \eqref{hw3}, or \eqref{hw4}
accordingly as $\hat{w}_i$ is of one of these types. The remaining side
will be regarded as being oriented in the opposite direction, so that its
label reads $\hat{u}$. 

We enlarge $\Delta$ to a
diagram $\Delta'$ by adjoining folded $r$--corridors, as follows. 
First, for each subword $\hat{w}_i$ of type \eqref{hw1}, build a folded
corridor $C_i$ whose top is labeled by $\hat{v}_i$ and whose bottom is
labeled by the reduced word in $a_{\nu_0}, b_{\nu_0}$ representing
$\hat{u}_i$ in $A_{\nu_0}$. Adjoin this corridor to $\Delta$ along the
side labeled by $\hat{v}_i$. Next, build a folded $r_{j'}$--corridor $C_0$
whose top is labeled by $\hat{u}$ and whose bottom is labeled by
$\hat{v}$. Adjoin it to $\Delta$ along the side labeled by
$\hat{u}$. The resulting diagram is $\Delta'$. 

Now let $\sigma \subset \widehat{T}$ be the maximal segment whose
endpoints correspond to the peripheral subgroups $A_{\nu_{j'}}$ and
$A_{\nu_{j''}}$. The corridor scheme $\Ss_{\sigma}$ defines
$\sigma$--corridors in $\Delta$. Every edge in the side labeled $\hat{u}$
has a $\sigma$--corridor emanating from it, landing on a side of type
\eqref{hw1} or \eqref{hw4}. (They cannot land on the other 
sides, because their edges are not members of $\Ss_{\sigma}$. They cannot
land on the side labeled $\hat{u}$, because $\hat{u}$ is reduced; cf. the
proof of Proposition \ref{balanced}.) Decompose $\hat{u}$ as $z_1 \dotsm
z_{\ell}$ such that 
\begin{itemize}
\item for each $z_i$, the $\sigma$--corridors emanating from the edges
  labeled by $z_i$ either land on a single side of type \eqref{hw1}, or
  on a union of sides of type \eqref{hw4} 
\item each $z_i$ is maximal with respect to the preceding property. 
\end{itemize}
Let $p_0, \dotsc, p_{\ell}$ be the vertices along the side labeled
$\hat{u}$ such that for each $i$, the arc labeled by $z_i$ has endpoints
$p_{i-1}$ and $p_i$. These vertices lie along the top of the folded 
$r_{j'}$--corridor $C_0$. Choose vertices $q_0, \dotsc, q_{\ell}$ along
the bottom of $C_0$ such that $p_i$ is nearly above $q_i$ for each $i$
and $q_0, q_{\ell}$ are the endpoints of the bottom. Let $[p_{i-1},p_i]$
and $[q_{i-1},q_i]$ denote the segments along the top and bottom,
respectively, with the indicated endpoints. 

Next define the index sets
\begin{align*}
I_{\eqref{hw1}} \ &= \ \{ i \mid \text{ the $\sigma$--corridors
  emanating from $z_i$ land on a side of type \eqref{hw1}} \},\\
I_{\eqref{hw4}} \ &= \ \{ i \mid \text{ the $\sigma$--corridors
  emanating from $z_i$ land on sides of type \eqref{hw4}} \}
\end{align*}
so that $I_{\eqref{hw1}} \cup I_{\eqref{hw4}} = \{ 1, \dotsc,
\ell\}$. For each $i \in I_{\eqref{hw1}}$ let $j_i$ be the index such
that the $\sigma$--corridors emanating from $z_i$ land on the side
labeled by $\hat{v}_{j_i}$. 

If $C$ and $C'$ are two $\sigma$--corridors emanating from $\hat{u}$ and
landing on the same side $\hat{v}_j$, then every corridor emanating from
$\hat{u}$ between $C$ and $C'$ must also land on $\hat{v}_j$, since
corridors do not cross. Moreover, if $C$ and $C'$ are adjacent in
$\hat{u}$ then they will land on adjacent edges of $\hat{v}_j$;
otherwise, any $\sigma$--corridor emanating from $\hat{v}_j$ between $C$
and $C'$ is forced to land on $\hat{v}_j$, contradicting that $\hat{v}_j$
is reduced. 

These remarks imply that for every $i \in I_{\eqref{hw1}}$, the corridors
emanating from $z_i$ land on a connected subarc $\alpha_i$ of the side
labeled $\hat{v}_{j_i}$, and in fact the label along $\alpha_i$ is the word
$z_i(a_{\nu_{j'}}, b_{\nu_{j'}})$ or $z_i(a_{\nu_{j''}}, b_{\nu_{j''}})$, by Remark
\ref{corridorremark}. Note that $\alpha_i$ lies along the top of the
$r$--corridor $C_{j_i}$. Let $\beta_i$ be a subarc of the bottom of
$C_{j_i}$ such that the endpoints of $\alpha_i$ are nearly above those of
$\beta_i$. Since $[p_{i-1},p_i]$ and $\alpha_i$ are labeled by the same
word, we have 
\begin{equation}\label{matchingsegments} 
\abs{ \abs{[q_{i-1},q_i]} - \abs{\beta_i}} \ \leq \ 2K_2 
\end{equation}
by Lemma \ref{segmentsnearlyabove}. Therefore, 
\begin{equation*}
\abs{[q_{i-1},q_i]} \ \leq \ \aabs{\hat{u}_{j_i}}_{A_{\nu_0}} + \ 2K_2.
\end{equation*}
Applying the induction hypothesis to $\hat{u}_{j_i}$ we obtain 
\begin{align}\label{I5term} 
\abs{[q_{i-1},q_i]} \ &\leq \ K_3 \aabs{\hat{u}_{j_i}}^{\alpha} \ + \ 2K_2
\notag \\
 &= \ K_3 (\aabs{\hat{w}_{j_i}} - 2)^{\alpha} \ + \ 2K_2. 
\end{align}

Next, if $i \in I_{\eqref{hw4}}$, we have
\begin{align}\label{I8term}
\abs{[q_{i-1}, q_i]} \ &\leq \ \aabs{\phi^{-n}(z_i)}_{A_{\nu_{j'}}} + \
K_2 \notag \\
&\leq \ \Lambda^{\! n} \aabs{z_i} \ + \ K_2 \notag \\
&\leq \ K_3 \aabs{z_i} \ + \ K_2 
\end{align}
by Lemma \ref{segmentsnearlyabove}. Now define the disjoint sets 
\begin{align*}
J_{\eqref{hw1}} \ &= \ \{ j_i \mid i \in I_{\eqref{hw1}} \}, \\
J_{\eqref{hw4}} \ &= \ \{ j \mid \hat{w}_{j} \text{ is of type }
\eqref{hw4} \}
\end{align*}
and note that $\abs{J_{\eqref{hw1}}} = \abs{I_{\eqref{hw1}}}$. Summing
the inequalities \eqref{I5term} and \eqref{I8term} over all $i \in
I_{\eqref{hw1}} \cup I_{\eqref{hw4}}$ we obtain
\[ \aabs{\hat{v}} \ = \ \sum_{i=1}^{\ell} \abs{q_{i-1}, q_i} \ \leq \
K_3\left(\sum_{i \in I_{\eqref{hw1}}} (\aabs{\hat{w}_{j_i}} -
    2)^{\alpha}\right) \ + \ K_3\left(\sum_{i \in I_{\eqref{hw4}}}
  \aabs{z_i} \right) \ + \ \left(2\abs{I_{\eqref{hw1}}} +
\abs{I_{\eqref{hw4}}}\right) K_2. \]
By considering $\sigma$--corridors we have $\sum_{i \in
  I_{\eqref{hw4}}} \aabs{z_i} \leq \sum_{j \in J_{\eqref{hw4}}}
\aabs{\hat{w}_j}$, and therefore 
\begin{equation}\label{firstmainsum}
\aabs{\hat{v}} \ \leq \ K_3\left(\sum_{j \in J_{\eqref{hw1}}}
  (\aabs{\hat{w}_{j}} - 2)^{\alpha}\right) \ + \ K_3\left(\sum_{j \in
    J_{\eqref{hw4}}} \aabs{\hat{w}_j} \right) \ + \
\left(2\abs{I_{\eqref{hw1}}} + \abs{I_{\eqref{hw4}}}\right) K_2.
\end{equation}
Now observe that no two adjacent indices can both be in
$I_{\eqref{hw4}}$, by the maximality property of the words $z_i$. It
follows that $\abs{I_{\eqref{hw4}}} \leq \abs{I_{\eqref{hw1}}} + 1$,
and since $K_3 \geq 3K_2 /2$, we have 
\begin{align*}
\left( 2 \abs{I_{\eqref{hw1}}} + \abs{I_{\eqref{hw4}}} \right)K_2 \ &\leq \
\left( 3 \abs{I_{\eqref{hw1}}} + 1 \right) K_2 \\
&\leq \ K_3 \left( 2 \abs{I_{\eqref{hw1}}} + 2 \right).
\end{align*}
Combining this with \eqref{firstmainsum} we obtain
\begin{align}\label{secondmainsum}
\aabs{\hat{v}} \ &\leq \ K_3\left(\, 2 \ + \sum_{j \in J_{\eqref{hw1}}}
  \left( (\aabs{\hat{w}_{j}} - 2)^{\alpha} + 2\right)\right) \ + \
K_3\left(\sum_{j \in J_{\eqref{hw4}}} \aabs{\hat{w}_j} \right) \notag \\
&\leq \ K_3 \left( \, 2^{\alpha} + \sum_{j \in J_{\eqref{hw1}}}
  \aabs{\hat{w}_j}^{\alpha} + \sum_{j \in J_{\eqref{hw4}}}
  \aabs{\hat{w}_j}^{\alpha}\right) \notag \\
&\leq \ K_3 \left( \, 2 + \sum_{j = 1}^{k'} \aabs{\hat{w}_j}
\right)^{\alpha} \ = \ \ K_3 \aabs{\hat{w}}^{\alpha}.
\end{align}
Finally, consider the words $w_L$ and $w_R$, and note that $v = w_L
\hat{v} w_R$. Choose any index $j \in \{1, \dotsc, m\}$ and apply
Proposition \ref{balanced}/Remark \ref{weightedremark} to obtain
\[ \aabs{g}_{A_{\nu_0}} \ \leq \ \aabs{v}_{\nu_j} + \aabs{v}_{\nu_0} +
\aabs{v}_x + \aabs{v}_y \ \leq \ \aabs{v} \ = \ \aabs{w_L} +
\aabs{\hat{v}} + \aabs{w_R}.\]
Combining this with \eqref{secondmainsum} yields 
\[ \aabs{g}_{A_{\nu_0}} \ \leq \ \aabs{w_L} + K_3 \aabs{\hat{w}}^{\alpha}
+ \aabs{w_R} \ \leq \ K_3 \aabs{w_L \hat{w} w_R}^{\alpha} \ = \ K_3
\aabs{w}^{\alpha}. \]
This completes the proof in Case II, and the proof of the proposition. 
\end{proof}

\begin{corollary}\label{distortioncor} 
Given $T$ and $n$ there is a constant $K_4$ such that if $w$ is a word in
the standard generators of $S_{T,n}$ representing an element $g \in
A_{\nu_0}$ then  
\[\abs{g}_{A_{\nu_0}} \ \leq \ K_4 \abs{w}^{\alpha}\]
where $\alpha = n \log_m(\lambda)$. 
\end{corollary}

\begin{proof}
One simply enlarges the constant $K_3$ to $K_4$, to account for the
maximum scaling factor between $\abs{ \, \cdot \,}$ and $\aabs{ \, \cdot
  \, }$.  
\end{proof}

\section{The Dehn function of $S_{T,n}$} \label{sec:dehn}

Before proceeding we need to establish some additional properties of
$\sigma$--corridors. First, we identify the subgroups of $V_T$ generated
by the labels along the sides of $\sigma$--corridors. Fixing $\sigma$,
this subgroup will be denoted $S_{\sigma}$ (the \emph{side
  subgroup}). Looking at the corridor cells in Figure \ref{fig:S-sigma},
it is clear that $S_{\sigma}$ is generated by the subgroups $F_j$ for
various $j$; namely, whenever $\sigma$ passes through a triangle in the
$(\abs{T}+2)$--gon with corners $i, j,k$, if $\sigma$ separates $i$ from
$j$ and $k$, then $F_j$ and $F_k$ are in $S_{\sigma}$. 

\begin{lemma}\label{sidelemma}
The side subgroup $S_{\sigma}$ is the free product of the factors $F_j
\times F_k$, with one factor for each triangle in the $(\abs{T}+2)$--gon
through which $\sigma$ passes. 
\end{lemma}

\begin{proof}
First, there is a surjective homomorphism from the free product onto
$S_{\sigma}$, induced by inclusion of the subgroups $F_j \times
F_k$. Injectivity then follows from Lemma \ref{basslemma}, once we
observe that in each vertex group $F_i \times F_j \times F_k$ of $V_T$,
the subgroup $F_j \times F_k$ has trivial intersection with the edge
groups $A_i \subset F_i \times F_j$ and $A_k \subset F_k \times F_i$. 
\end{proof}

Next we need some observations about corridors in diagrams over the
subgroup $F_0 \times F_1 \subset V$. Let $X_{01} \subset X$ be the
subcomplex whose $2$--cells have boundary edges labeled by the elements
$x_0, y_0, x_1, y_1, a_0, b_0$. It has two quadrilateral cells and four
triangular cells (see Figure \ref{fig:S-sigma}). Its fundamental group is
$F_0 \times F_1$. Define two orientable corridor schemes $\Ss_0, \Ss_1$
over $X_{01}$ as follows: $\Ss_0$ contains the edges labeled $x_0, y_0,
a_0, b_0$ and $\Ss_1$ contains the edges labeled $x_1, y_1, a_0,
b_0$. Note that the \emph{side subgroup} associated with $\Ss_0$ is
$F_1$, and the side subgroup of $\Ss_1$ is $F_0$. 

\begin{remarks}\label{sideremark} 
Let $\Delta$ be a reduced diagram over $X_{01}$. 

(1) Let $\alpha$ be a path in the $1$--skeleton of $\Delta$ along the
side of an $\Ss_0$--corridor, and let $\beta$ be a path along the side of
an $\Ss_1$--corridor. Then $\alpha$ and $\beta$ intersect in at most one
point. If there were two points in the intersection, then the labels
along $\alpha$ and $\beta$ between these points give non-trivial elements
of $F_1$ and $F_0$ respectively, representing the same element of $F_0
\times F_1$. 

(2) Every $2$--cell of $\Delta$ is contained in both an $\Ss_0$-- and an 
$\Ss_1$--corridor. This is immediate by examining the corridor cells for
$\Ss_0$ and $\Ss_1$. 

(3) There are no $\Ss_0$-- or $\Ss_1$--corridors of annulus type (because
the side subgroups are free and $\Delta$ is reduced). 
\end{remarks}

\begin{proposition}[Area in $V_T$]\label{areainVT}
Given $T$ there is a constant $K_5$ with the following property. 
Suppose $w$ and $z(a_{\nu_{\ell}}, b_{\nu_{\ell}})$ represent the same element of
$A_{\nu_{\ell}} \subset V_T$, where $z$ is reduced and $w$ decomposes as $w_1
\dotsm w_k$ where each $w_i$ is a reduced word in the generators $x_j,
y_j$ for some $j$, or the peripheral generators $a_{\nu_j}, b_{\nu_j}$
for some $j$. Let $\Is$ be the set of pairs of indices $(i,j)$ such that
either $i\not= j$ or $i=j$ and $w_i$ is not a word in peripheral
generators of $V_T$. Then 
\[\area(wz^{-1}) \ \leq \ K_5 \sum_{(i,j) \in \Is} \abs{w_i}\abs{w_j}. \]
\end{proposition}

\begin{remark}\label{arearemark}
Let $\Ps$ be the set of indices $i$ such that $w_i$ is a word in
peripheral generators of $V_T$. Then 
\[ \sum_{(i,j) \in \Is} \abs{w_i}\abs{w_j} \ + \ \sum_{i \in \Ps}
\abs{w_i}^2 \ = \ \abs{w}^2.\] 
\end{remark}

\begin{proof}[Proof of Proposition \ref{areainVT}] 
Recall from Remark \ref{VTisCAT0} that $V_T$ is CAT(0), and therefore its
Dehn function is quadratic. Thus there is a constant $C \geq 1$
(independent of $z, w$) such that $\area(wz^{-1}) \leq C (\abs{w} +
\abs{z})^2$. Also, $\abs{z} \leq \abs{w}$ by Proposition \ref{balanced},
so $\area(wz^{-1}) \leq 4C \abs{w}^2$. 

The proof now falls into two cases: the generic case and a special case
which is somewhat more difficult. The latter case is when there is an
index $i'$ such that $w_{i'}$ is a word in the generators
$a_{\nu_{\ell}}, b_{\nu_{\ell}}$ (the same as $z$) and $\abs{w_{i'}} >
(1/2) \abs{w}$. 

\medskip

Consider first the generic case (that is, whenever $w_i$ is a word in the
generators $a_{\nu_{\ell}}, b_{\nu_{\ell}}$ we have $\abs{w_i}
\leq (1/2) \abs{w}$). We claim that $\abs{w_i} \leq (1/2) 
\abs{w}$ for \emph{every} $i \in \Ps$. To see this, suppose $w_i$ is a
word in $a_{\nu_j}, b_{\nu_j}$ with $j \not= \ell$ and let $\nu_{j'}$ be
a peripheral index not equal to $\nu_j$ or $\nu_{\ell}$. Let $\hat{w}$ be
the complement of $w_i$ in the cyclic word $wz^{-1}$. Apply Proposition
\ref{balanced} to obtain 
\[ \abs{w_i} \ \leq \ \abs{\hat{w}}_{\nu_{j'}} + \abs{\hat{w}}_{\nu_j} + 
\abs{\hat{w}}_x + \abs{\hat{w}}_y.\]
The right hand side counts no letters of $z$ because $j, j' \not=
\ell$. Thus, $\abs{w_i} \leq (1/2) \abs{w}$. 

This claim implies (see Figure \ref{fig:boxes}) that 
\begin{equation}\label{boxes} 
\sum_{i \in \Ps} \abs{w_i}^2 \ \leq \ \sum_{i\not= j}
\abs{w_i}\abs{w_j}.
\end{equation}
\begin{figure}[ht]
\labellist
\hair 2pt
\small
\pinlabel* {$\abs{w_i}$} at 7.5 41.5
\endlabellist
\includegraphics[width=1.5in]{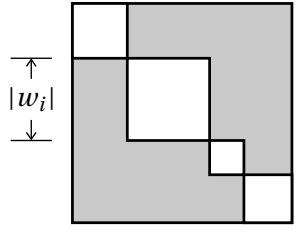}
\caption{The area of the squares does not exceed the shaded area, if the
  side lengths are each at most half the total side
  length.}\label{fig:boxes} 
\end{figure}

From Remark \ref{arearemark} we deduce that 
\[\sum_{i\in\Ps} \abs{w_i}^2 \ \leq \ (1/2) \abs{w}^2\]
and therefore
\[ \sum_{(i,j) \in \Is} \abs{w_i}\abs{w_j} \ \geq \ (1/2)
\abs{w}^2.\] 
Finally, we have 
\[\area(wz^{-1}) \ \leq \ 4C\abs{w}^2 \ \leq \ 8C \sum_{(i,j) \in \Is}
\abs{w_i}\abs{w_j}\]
so we are done by taking $K_5 \geq 8C$. 

\medskip

Next consider the special case: $w_{i'}$ is a word in $a_{\nu_{\ell}},
b_{\nu_{\ell}}$ and $\abs{w_{i'}} > (1/2) \abs{w}$. Write $w = w_L w_{i'}
w_R$, so that $\abs{w_L} + \abs{w_R} < \abs{w_{i'}}$. Note that it will
suffice for us to prove that $\area(wz^{-1}) \ \leq \
K_5 \abs{w_{i'}}(\abs{w_L} + \abs{w_R})$. 

Let $\sigma$ and $\sigma'$ be maximal segments in $\widehat{T}$ with
common endpoint $v_{\nu_{\ell}}$, which diverge immediately. That is,
their intersection consists of the single edge from the leaf
$v_{\nu_{\ell}}$ to its parent vertex in $T$. This edge lies inside a
triangle in the $(\abs{T}+2)$--gon. Suppose without loss of generality
that this triangle has index triple $(0, 1, 2)$ and that $A_{\nu_{\ell}}
= A_0 \subset F_0 \times F_1$. The vertex group in $V_T$ corresponding to
this triangle is $F_0 \times F_1 \times F_2$; denote this subgroup by
$V$. For concreteness, suppose that $\sigma$ separates corner $0$ from
corners $1$ and $2$, and $\sigma'$ separates corner $1$ from corners $2$
and $0$. 

We can express the $(\abs{T}+2)$--gon as a union of two smaller
sub-diagrams whose intersection is the $(0,1,2)$ triangle. Note that
$\sigma$ and $\sigma'$ each lie wholly inside one of these
sub-diagrams. Thus $V_T$ has an expression as $A \ast_V B$ where $A$ is
the fundamental group of the sub-diagram containing $\sigma$ and $B$ is
the fundamental group of the sub-diagram containing $\sigma'$. See Figure
\ref{fig:AunionB}. 
\begin{figure}[ht]
\labellist
\hair 2pt
\large
\pinlabel* {$=$} at 95 32
\pinlabel* {$\cup$} at 178 32

\normalsize
\pinlabel {${\sigma}$} [Br] at 6 18.5
\pinlabel {${\sigma}$} [tr] at 46 3.3
\pinlabel {${\sigma'}$} [tl] at 49 6
\pinlabel {${\sigma'}$} [Bl] at 74 15

\pinlabel {${\sigma}$} [Br] at 112.5 20
\pinlabel {${\sigma}$} [t] at 154 3.5
\pinlabel {${\sigma'}$} [t] at 205 6
\pinlabel {${\sigma'}$} [l] at 230 19

\small
\pinlabel* {$\textcolor{blue}{0}$} at 40.5 12
\pinlabel* {$\textcolor{blue}{1}$} at 55.5 12
\pinlabel* {$\textcolor{blue}{2}$} at 48 24

\pinlabel* {$\textcolor{blue}{0}$} at 147 12
\pinlabel* {$\textcolor{blue}{1}$} at 160 12
\pinlabel* {$\textcolor{blue}{2}$} at 154 24

\pinlabel* {$\textcolor{blue}{0}$} at 197 12
\pinlabel* {$\textcolor{blue}{1}$} at 209.5 12
\pinlabel* {$\textcolor{blue}{2}$} at 203.5 24

\endlabellist
\includegraphics[width=4.4in]{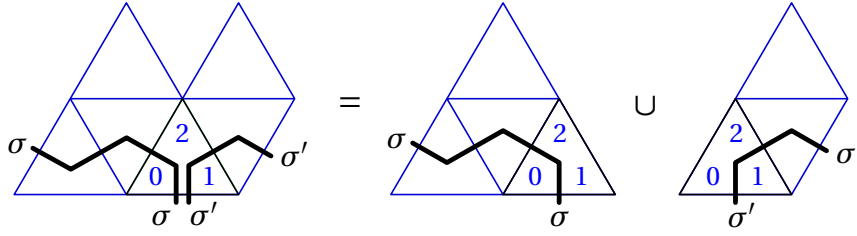}
\caption{Expressing $V_T$ as $A \ast_V B$ with $S_{\sigma} \subset A$ and 
  $S_{\sigma'} \subset B$.}\label{fig:AunionB} 
\end{figure}
In particular,
$S_{\sigma} \subset A$, $S_{\sigma'} \subset B$, and $A \cap B = V$. 

Let $\Delta$ be a reduced van Kampen diagram over $X_T$ with boundary
$wz^{-1}$. Think of the boundary as being two arcs with the same
endpoints, labeled by $w$ and $z$ respectively. Every edge along $w_{i'}$
has a $\sigma$--corridor emanating from it. Since $w_{i'}$ and $z$ are
reduced words in $a_{\nu_{\ell}}, b_{\nu_{\ell}}$, we can argue (as
usual) that no $\sigma$--corridor has both ends on $w_{i'}$ or on $z$, 
and that the $\sigma$--corridors emanating from $w_{i'}$ and landing on
$z$ all land on a connected subarc of the arc labeled $z$. The
$\sigma$--corridors not landing on $z$ must land on $w_L$ or $w_R$, and
one finds that these comprise at most $1/3$ of the corridors emanating
from $w_{i'}$ (because $\abs{w_{i'}} > \abs{w_L} + \abs{w_R}$). 

Let $p_1, p_2$ be the initial and final endpoints of the maximal segment
along $w_{i'}$ whose $\sigma$--corridors land on $z$. Let $q_1, q_2$ be
the analogous points along $z$, so that $p_i$ is joined to $q_i$ by the
side of a $\sigma$--corridor ($i=1,2$). Note that the subsegments $[p_1,
p_2]$ and $[q_1, q_2]$ are labeled by the same word in $a_{\nu_{\ell}},
b_{\nu_{\ell}}$, by Remark \ref{corridorremark}. 

In a similar fashion, there are $\sigma'$--corridors emanating from
$w_{i'}$, at least 2/3 of which land on a connected subsegment of
$z$. Define $p_1'$, $p_2'$, $q_1'$, $q_2'$ analogously to $p_1$, $p_2$,
$q_1$, and $q_2$. 

Now let $[p_i,q_i]$ denote the side of the $\sigma$--corridor 
joining $p_i$ to $q_i$ ($i = 1,2$). Define $[p_i', q_i']$
analogously. Note that the label along $[p_1,q_1]$ represents an element
of $S_{\sigma} \subset A$. However, $p_1$ and $q_1$ are also joined by
the path $[p_1, p_1'] \cdot [p_1', q_1'] \cdot [q_1', q_1]$ where the
first and third segments run along $w_{i'}$ and $z$ respectively. This
path represents an element of $B$. Therefore, $[p_1,q_1]$ represents an
element of $A \cap B = V$. By Lemma \ref{sidelemma} we have $S_{\sigma}
\cap V = F_1 \times F_2$, and so $[p_1,q_1]$ in fact represents an
element of this latter subgroup. By the same argument, $[p_2, q_2]$ also
represents an element of $F_1 \times F_2$. 

We need to introduce a little more notation. Let $\alpha$ be the path
along the boundary of $\Delta$ from $q_1$ to $p_1$ which contains the
segment labeled $w_L$. Similarly, let $\beta$ be the path in the
boundary from $p_2$ to $q_2$ which contains $w_R$. Let $p_0, p_3$ be the
initial and final endpoints of the segment labeled $w_{i'}$. Let $q_0,
q_3$ be the endpoints of the segment labeled $z$. Note that
$\abs{[p_0,p_1]} \leq \abs{w_L}$ since the $\sigma$--corridors from
$[p_0, p_1]$ land on $w_L$. Similarly, $\abs{[p_2, p_3]} \leq
\abs{w_R}$. Because $\abs{z} \leq \abs{w}$, we have
\begin{align*}
\abs{[q_0, q_1]} + \abs{[q_2, q_3]} \ &\leq \ \abs{w_L} + \abs{[p_0,
  p_1]} + \abs{[p_2, p_3]} + \abs{w_R} \notag \\
&\leq \ 2(\abs{w_L} + \abs{w_R}) 
\end{align*}
and therefore 
\begin{equation}\label{alphabeta} 
\abs{\alpha} + \abs{\beta} \ \leq \ 4(\abs{w_L} + \abs{w_R}).
\end{equation}

At this point we discard the diagram $\Delta$ and start over with its
boundary loop. We have seen that $\alpha$ and $\beta$ represent 
elements $g_{\alpha}$ and $g_{\beta}$ of the subgroup $F_1 \times F_2
\subset V$. Attach a segment $[q_1, p_1]$ to the points $q_1$ and $p_1$,
and label it by a reduced word representing $g_{\alpha}$, of the
form $u(x_1,y_1)v(x_2,y_2)$. Similarly, attach a segment $[q_2, p_2]$
labeled by a reduced word $u'(x_1, y_1) v'(x_2, y_2)$ representing
$g_{\beta}^{-1}$ to the points $q_2$ and $p_2$. Note that the loop $[p_1,
p_2] \cdot [p_2, q_2] \cdot [q_2, q_1] \cdot [q_1, p_1]$ is labeled by
generators of $V$, and the only occurrences of generators of $F_2$ are
in the reduced words $v$ and $v'$. It follows that $v = v'$. 

Now define $o_1$ to be the point along $[q_1, p_1]$ such that $[q_1,
o_1]$ is labeled by $u(x_1, y_1)$ and $[o_1, p_1]$ is labeled by $v(x_2,
y_2)$. Similarly let $o_2$ be the point on $[q_2, p_2]$ such that $[q_2,
o_2]$ is labeled by $u'(x_1, y_1)$ and $[o_2, p_2]$ is labeled by $v(x_2,
y_2)$. Attach one more segment $[o_1, o_2]$ to the points $o_1$ and
$o_2$, labeled by the same word in $a_0, b_0$ ($= a_{\nu_{\ell}},
b_{\nu_{\ell}}$) as $[p_1,p_2]$ and $[q_1, q_2]$. 

We proceed now to fill the original boundary loop with a van Kampen
diagram, in four parts. Fill the loop $\alpha \cdot [p_1, q_1]$ with a
least-area van Kampen diagram $\Delta_{\alpha}$ over $X_T$. Fill $\beta
\cdot [q_2, p_2]$ with a least-area diagram $\Delta_{\beta}$ over $X_T$. 
To estimate the areas of $\Delta_{\alpha}$ and $\Delta_{\beta}$, choose
a segment $\sigma'' \subset \widehat{T}$ which passes through the
$(0,1,2)$ triangle and separates corner $2$ from corners $0$ and
$1$. Note that $x_1, y_1 \in \Ss_{\sigma'} - \Ss_{\sigma''}$ and $x_2,
y_2 \in \Ss_{\sigma''} - \Ss_{\sigma'}$. Every edge along $[p_1,q_1]$ has
either a $\sigma'$--corridor or a $\sigma''$--corridor in
$\Delta_{\alpha}$ emanating from it, and not both. Since the words
$u(x_1, y_1)$, $v(x_2, y_2)$ are reduced, these corridors can only land
on $\alpha$. It follows that 
\begin{equation}\label{alpha}
\abs{[p_1, q_1]} \ \leq \ 2\abs{\alpha}. 
\end{equation}
By a similar argument using $\Delta_{\beta}$, we have
\begin{equation}\label{beta}
\abs{[p_2, q_2]} \ \leq \ 2\abs{\beta}. 
\end{equation} 
From \eqref{alpha} and \eqref{beta} we deduce that
\begin{equation*}
\area(\Delta_{\alpha}) \ \leq \ C (3\abs{\alpha})^2 \ \text{ and } \
\area(\Delta_{\beta}) \ \leq \ C (3 \abs{\beta})^2. 
\end{equation*}

Two more loops remain to be filled. Let $\hat{w}(a_0, b_0)$ be the word
labeling $[p_1, p_2]$ and $[o_1, o_2]$. The loop $[p_1, p_2] \cdot [p_2,
o_2] \cdot [o_2, o_1] \cdot [o_1, p_1]$ is labeled by the commutator
$[\hat{w}(a_0, b_0), v(x_2, y_2)]$. Adjoin a triangular relator to each
edge of $[p_1, p_2]$ and $[o_1, o_2]$ to obtain paths labeled by the
word $\hat{w}(\overline{x}_0 x_1, \overline{y}_0 y_1)$. Now these paths
and the paths $[p_1, o_1]$, $[p_2, o_2]$ can be filled using commutator
relators. In this way the loop $[p_1, p_2] \cdot [p_2, o_2] \cdot [o_2,
o_1] \cdot [o_1, p_1]$ bounds a diagram $\Delta_{012}$ over $X_T$ of area
$2\abs{[p_1, p_2]} + 2 (2\abs{[p_1, p_2]})\abs{v}$. 

Finally consider the loop $[o_1, o_2] \cdot [o_2, q_2] \cdot [q_2, q_1]
\cdot [q_1, o_1]$. It is labeled entirely by generators of $F_0 \times
F_1$ and represents the trivial element, so it bounds a least-area
diagram $\Delta_{01}$ over $X_{01}$. Its boundary is labeled by the word
\[\hat{w}(a_0, b_0) u'(x_1, y_1)^{-1} \hat{w}(a_0, b_0)^{-1} u(x_1,
y_1)\] 
with each of the four subwords being reduced. Consider corridors in
$\Delta_{01}$ for the corridor schemes $\Ss_0$ and $\Ss_1$. The
$\Ss_0$--corridors can only land on the sides $[q_1, q_2]$ and $[o_1,
o_2]$, and since these sides are labeled by reduced words, each such
corridor has one end on each side. Thus, each edge of $[q_1,q_2]$ is
joined by an $\Ss_0$--corridor to the edge of $[o_1, o_2]$ corresponding
to the same letter of $\hat{w}$. 

Consider the arrangement of the $\Ss_1$--corridors in
$\Delta_{01}$. Every edge in the boundary has an $\Ss_1$--corridor
emanating from 
it. There are no corridors of annulus type (see Remarks
\ref{sideremark}), and corridors cannot join two edges in the same side
of the boundary. If there is a corridor joining $[q_1, o_1]$ to $[q_2,
o_2]$ then we must have 
\begin{equation}\label{short}
2\abs{[p_1,p_2]} \ \leq \ \abs{[q_1, o_1]} + \abs{[q_2, o_2]}
\end{equation}
and we will be satisfied for the moment. Assume now that no such corridor
is present. Then, the most general arrangement is shown in Figure
\ref{fig:corridor-config2}. 
\begin{figure}[ht]
\labellist
\hair 2pt
\small
\pinlabel* {$m$} at 24 6
\pinlabel {$o_1$} [r] at 0 60
\pinlabel {$q_1$} [r] at .5 14
\pinlabel {$o_2$} [l] at 128.5 65
\pinlabel {$q_2$} [l] at 127 1
\endlabellist
\includegraphics[width=2.5in]{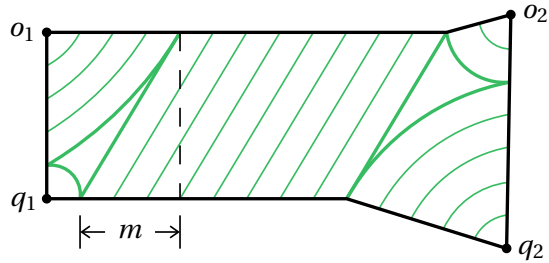}
\caption{A generic configuration of
  $\Ss_1$--corridors in $\Delta_{01}$.}\label{fig:corridor-config2} 
\end{figure}

Notice that corridors running between $[o_1, o_2]$ and $[q_1, q_2]$ will
land on edges that are offset by a fixed amount $m \leq
\min\{\abs{[q_1,o_1]}, \abs{[q_2,o_2]}\}$. In particular an
$\Ss_0$--corridor can be crossed by at most $m+1$ $\Ss_1$--corridors of
this type. There are $\abs{[q_1, o_1]} + \abs{[q_2,o_2]}$ 
$\Ss_1$--corridors not of this type, so an $\Ss_0$--corridor can cross no
more than $2(\abs{[q_1, o_1]} + \abs{[q_2,o_2]})$ $\Ss_1$--corridors
overall. 

By Remarks \ref{sideremark}, an $\Ss_0$--corridor and an
$\Ss_1$--corridor will intersect if and only if their endpoints are linked
(or equal) in the boundary of $\Delta_{01}$, and when this occurs, the
intersection will have area $1$ or $2$. Moreover, every $2$--cell of
$\Delta_{01}$ is in the intersection of such a pair. Since every
$\Ss_0$--corridor is crossed by at most $2(\abs{[q_1, o_1]} +
\abs{[q_2,o_2]})$ $\Ss_1$--corridors, we conclude that 
\begin{equation*}
\area(\Delta_{01}) \leq
4(\abs{[q_1, o_1]} + \abs{[q_2,o_2]})\abs{[o_1, o_2]}.
\end{equation*}

To finish, first suppose we are in the special case
where \eqref{short} holds. Then, the total perimeter is at most
$12(\abs{w_L} + \abs{w_R})$ by \eqref{alpha}, \eqref{beta}, and
\eqref{alphabeta}. Then 
\[\area(wz^{-1}) \ \leq \ C (12(\abs{w_L} + \abs{w_R}))^2 \ \leq \ 144C
\abs{w_{i'}} (\abs{w_L} + \abs{w_R})\]
so we require $K_5 \geq 144C$ to cover this case. 

Otherwise, we use our estimates for the areas of the four diagrams
$\Delta_{\alpha}$, $\Delta_{\beta}$, $\Delta_{012}$, and
$\Delta_{01}$. Since $\abs{[p_1,p_2]} = \abs{[o_1, o_2]} = \abs{[q_1,
  q_2]} \leq \abs{w_{i'}}$, these estimates yield 
\[\area(wz^{-1}) \ \leq \ 9C\abs{\alpha}^2 + 9C\abs{\beta}^2 +
(2\abs{w_{i'}} + 4\abs{w_{i'}}\abs{v}) + 4(\abs{u} + \abs{u'})
\abs{w_{i'}}.\] 
Using \eqref{alpha} and \eqref{beta} this reduces to
\[\area(wz^{-1}) \ \leq \ 9C(\abs{\alpha} + \abs{\beta})^2 + \abs{w_{i'}}
(2 + 8 (\abs{\alpha} + \abs{\beta})).\]
Note that $(\abs{\alpha} + \abs{\beta}) \leq 4\abs{w_{i'}}$ by
\eqref{alphabeta}, and a further application of \eqref{alphabeta} yields 
the following:
\begin{align*}
\area(wz^{-1}) \ &\leq \ 36C \abs{w_{i'}} (\abs{\alpha} + \abs{\beta}) +
\abs{w_{i'}} (2 + 8 (\abs{\alpha} + \abs{\beta})) \\
&\leq \ (36C + 10) \abs{w_{i'}}(\abs{\alpha} + \abs{\beta}) \\
&\leq \ (144C + 40) \abs{w_{i'}} (\abs{w_L} + \abs{w_R}). 
\end{align*}
Taking $K_5 \geq 144C + 40$ completes the proof. 
\end{proof}

\begin{proposition}[Area in $S_{T,n}$]\label{areainS}
Given $T$ and $n$ there is a constant $K_6$ with the following
property. Suppose $w$ and $z(a_{\nu_j},b_{\nu_j})$ represent the same
element of $A_{\nu_j} \subset S_{T,n}$, where $w$ is a word in the
standard generators of $S_{T,n}$ and $z$ is reduced. Then $\area(wz^{-1})
\ \leq \ K_6 \abs{w}^{2\alpha}$.  
\end{proposition}

We follow the proof of Proposition 5.5 in \cite{BBFS}. 

\begin{proof}
The proof is by induction on $\abs{w}$. Let $M = \max\{ \abs{\phi^n(x)},
\abs{\phi^n(y)}, \abs{\phi^{-n}(x)}, \abs{\phi^{-n}(y)}\}$.
Let $K_6 = M^2 K_4^2 K_5$, where $K_4$ is given by Corollary
\ref{distortioncor} and $K_5$ is given by Proposition
\ref{areainVT}. Write $w$ as $w_1 \dotsc w_k$ where each $w_i$ 
either is a word in the standard generators of $V_T$, or $w_i =
r_{j}^{\pm 1} u_i r_{j}^{\mp 1}$ for some $j$. Let $I_r$ be the set of
indices for which the latter case occurs, and note that $w_i$ represents
an element of a peripheral subgroup of $V_T$. Let $v_i$ be the reduced
word in the generators of that subgroup representing $w_i$. For $i
\not\in I_r$ let $v_i = w_i$, and define $v = v_1 \dotsm v_k$. 

By Proposition \ref{areainVT} we have $\area(vz^{-1}) \leq K_5
\sum_{(i,j) \in \Is} \abs{v_i} \abs{v_j}$. For $i \in I_r$ we have either
$\abs{v_i} \leq K_4 \abs{w_i}^{\alpha}$ (if $v_i \in A_{\nu_0}$) or
$\abs{v_i} \leq M K_4 \abs{u_i}^{\alpha}$ (if $u_i \in A_{\nu_0}$), by
Corollary \ref{distortioncor}. In any case (including $i \not\in I_r)$ we
have $\abs{v_i} \leq M K_4 \abs{w_i}^{\alpha}$. Therefore, 
\begin{equation}\label{middlearea}
\area(vz^{-1}) \ \leq \ K_6 \sum_{(i,j)\in\Is} \abs{w_i}^{\alpha}
\abs{w_j}^{\alpha}. 
\end{equation}

Next note that 
\[\area(wv^{-1}) \ \leq \ \sum_{i \in I_r} \area(w_i v_i^{-1})\]
since $w_i = v_i$ for $i \not\in I_r$. For each term
$\area(w_i v_i^{-1})$, recall that $w_i = r_j^{\pm 1} u_i r_j^{\mp
  1}$. Let $z_i$ be the reduced word in peripheral generators
representing the same peripheral element as $u_i$. Apply the induction
hypothesis to $u_i$ to obtain 
\begin{equation}\label{IHarea}
\area(u_i z_i^{-1}) \ \leq \ K_6 \abs{u_i}^{2\alpha} \ = \ K_6 (\abs{w_i} -
2)^{2\alpha}.
\end{equation}
One of the words $z_i, v_i$ is an element of $A_{\nu_0}$, so there is a
folded $r_j$--corridor with boundary word $r_j^{\pm 1} z_i r_j^{\mp 1}
v_i^{-1}$, where one of the boundary arcs labeled $z_i$ or $v_i$ is the
bottom and the other is the top. The area of this corridor is the length
of the bottom, which is at most $M\abs{v_i}$. As noted above, $\abs{v_i}
\leq MK_4\abs{w_i}^{\alpha}$, and so 
\begin{equation*}
\area(r_j^{\pm 1} z_i r_j^{\mp 1} v_i^{-1}) \ \leq \ MK_4
\abs{w_i}^{\alpha}. 
\end{equation*}
Together with \eqref{IHarea} we obtain 
\begin{align}\label{outerarea}
\area(w_i v_i^{-1}) \ &\leq \ K_6 ( (\abs{w_i} - 2)^{2\alpha} +
\abs{w_i}^{\alpha}) \notag \\
&\leq \ K_6 \abs{w_i}^{2\alpha} 
\end{align}
for $i \in I_r$. The last inequality above holds exactly as in
\cite[Proposition 5.5]{BBFS}: for numbers $x \geq 0$ one has
$(x+2)^{2\alpha} \geq x^{\alpha}(x+2)^{\alpha} + 2^{\alpha}(x+2)^{\alpha}
\geq x^{2\alpha} + (x+2)^{\alpha}$. 

Lastly, add together \eqref{middlearea} and \eqref{outerarea} for each $i
\in I_r$ to obtain
\begin{align*}
\area(wz^{-1}) \ &\leq \ K_6 \sum_{(i,j)\in \Is} \abs{w_i}^{\alpha}
\abs{w_j}^{\alpha} \ + \ K_6 \sum_{i \in  I_r} \abs{w_i}^{2\alpha} \\
&\leq \ K_6 \sum_{i \not= j} \abs{w_i}^{\alpha}
\abs{w_j}^{\alpha} \ + \ K_6 \sum_{i} \abs{w_i}^{2\alpha}. 
\end{align*}
The latter quantity is $K_6 \abs{w}^{2\alpha}$, as desired. 
\end{proof}

\begin{theorem}\label{dehnfunction}
Given $T$ and $n$ let $m = \abs{T} + 1$, let $\lambda > 1$ be the
Perron-Frobenius eigenvalue of $\phi$, and let $\alpha = n
\log_m(\lambda)$. If $\alpha \geq 1$ then the Dehn function of $S_{T,n}$
is given by $\delta(x) = x^{2\alpha}$. 
\end{theorem}

\begin{proof}
First we establish the lower bound $\delta(x) \succeq x^{2\alpha}$. Let
$w(x,y)$ be a  
monotone palindromic word (eg. $x$) and consider the snowflake diagrams
$\Delta (w, i)$ for $i \geq 1$. Let $n_i$ be the boundary length of
$\Delta(w, i)$. The boundary word has $4m^{i-1}$ occurrences of the
letters $r_j$ in it, and two adjacent such letters are never separated by
more than $m\abs{w}$ letters from $V_T$. Thus we have
\begin{equation}\label{ni}
4 m^{i-1} \ \leq \ n_i \ \leq \ 4(m\abs{w} + 1)m^{i-1}.
\end{equation}
From the second of these inequalities we obtain
\begin{align*}
(n_i)^{\alpha} \ &\leq \ \left( \frac{4(m\abs{w}+1)}{m}\right)^{\alpha}
(m^{\alpha})^i \notag \\
&= \ \left( \frac{4(m\abs{w}+1)}{m}\right)^{\alpha}\lambda^{ni}
\end{align*}
and so 
\begin{equation*}
\left(\frac{m}{4(m\abs{w}+1)}\right)^{\alpha} (n_i)^{\alpha} \ \leq \
\lambda^{ni}. 
\end{equation*}
Next, $\Delta(w,i)$ has area at least $6 \abs{T} \abs{\phi^{in}(w)}^2$,
which is the area of the doubled canonical diagram at its center. There
is a constant $C$ such that $\abs{\phi^k(v)} \geq C \lambda^k \abs{v}$
for every non-trivial word $v$. Thus,
\begin{align*}
\area(\Delta(w,i)) \ &\geq \ 6 \abs{T} C^2 \abs{w}^2
\bigl(\lambda^{ni}\bigr)^2 \\
&\geq \ 6\abs{T} C^2 \abs{w}^2
\left(\frac{m}{4(m\abs{w}+1)}\right)^{2\alpha}  (n_i)^{2\alpha}. 
\end{align*}
Taking $D = 6\abs{T} C^2 \abs{w}^2
\left(\frac{m}{4(m\abs{w}+1)}\right)^{2\alpha}$, we have shown that
$\delta_{Y_{T,n}}(n_i) \geq D(n_i)^{2\alpha}$ for each $i$, because
$\Delta(w,i)$ is a least-area diagram over $Y_{T,n}$ with boundary
length $n_i$. By \eqref{ni} the ratios $n_{i+1}/n_i$ are bounded, and so
we conclude by Remark \ref{equivremark} that $\delta(x)\succeq x^{2\alpha}$. 

The upper bound follows immediately from Proposition \ref{areainS}:
taking $z$ to be the empty word, $\area(w) \leq K_6 \abs{w}^{2\alpha}$
for every word $w$ representing the trivial element of $S_{T,n}$. Thus
$\delta(x) \preceq x^{2\alpha}$. 
\end{proof}


\newcommand{\etalchar}[1]{$^{#1}$}
\def\cprime{$'$}
\providecommand{\bysame}{\leavevmode\hbox to3em{\hrulefill}\thinspace}
\providecommand{\MR}{\relax\ifhmode\unskip\space\fi MR }
\providecommand{\MRhref}[2]{%
  \href{http://www.ams.org/mathscinet-getitem?mr=#1}{#2}
}
\providecommand{\href}[2]{#2}

\end{document}